\input preprint 
\def\preprint{}

\ifPreprint{
  \documentclass[reqno]{amsart}
}{
  \documentclass[reqno]{acmart}
  \citestyle{acmauthoryear}
}

\input macros
\ifPreprint{
  \usepackage[foot]{amsaddr}

\title[Fixed-point elimination in the \Ipc]{
  Fixed-point elimination \\ in the \capIPC \\
  (extended version)
} 
\thanks{This is an extended version of the conference paper \cite{fossacs}}

\author[Ghilardi]{Silvio Ghilardi}
\address{
  Silvio Ghilardi,
  Dipartimento di Matematica, Universit\`a
    degli Studi di Milano
  }
\email{silvio.ghilardi@unimi.it}

\author[Gouveia]{Maria Jo\~{a}o Gouveia} 
\address{
  Maria Jo\~{a}o Gouveia,
  Faculdade de Ci\^encias da Universidade de Lisboa,
  Portugal
}
\email{mjgouveia@fc.ul.pt}

\author[Santocanale]{Luigi Santocanale}
\address{
  Luigi Santocanale,
  LIS, CNRS UMR 7020, Aix-Marseille Universit\'e,
  France
}
\email{luigi.santocanale@lif.univ-mrs.fr}

}{
  \title[Fixed-point elimination in the \IPC]{Fixed-point elimination in the \IPC} 
\thanks{This is an extended version of the conference paper \cite{fossacs}}

\author[Ghilardi]{Silvio Ghilardi}
\affiliation{
  \institution{Dipartimento di Matematica, Universit\`a
    degli Studi di Milano}
  \country{Italy}
}
\email{silvio.ghilardi@unimi.it}

\author[Gouveia]{Maria Jo\~{a}o Gouveia} 
\affiliation{
  \institution{
    Faculdade de Ci\^encias da Universidade de Lisboa}
  \country{Portugal}
}
\email{mjgouveia@fc.ul.pt}

\author[Santocanale]{Luigi Santocanale}
\orcid{0000-0002-4237-7856}
\affiliation{
  \institution{
    LIS, CNRS UMR 7020, Aix-Marseille Universit\'e}
  \country{France}
}
\email{luigi.santocanale@lif.univ-mrs.fr}

\acmJournal{TOCL}

}

\begin{document}
\maketitle

\begin{abstract}
  It is a consequence of existing literature that \lgfp{s} of monotone
polynomials on \Ha{s}---that is, the algebraic models of the
\IPC---always exist, even when these algebras are not complete as
lattices.  The reason is that these \efp{s} are definable by formulas
of the \Ipc. Consequently, the $\mu$-calculus based on intuitionistic
logic is trivial, every $\mu$-formula being equivalent to a
fixed-point free formula. We give in this paper an axiomatization of
\lgfp{s} of formulas, and an algorithm to compute a \fpf formula
equivalent to a given $\mu$-formula. The axiomatization of the \gfp is
simple. The axiomatization of the \lfp is more complex, in particular
every monotone formula converges to its \lfp by Kleene's iteration in
a finite number of steps, but there is no uniform upper bound on the
number of iterations. We extract, out of the algorithm, upper bounds
for such $n$, depending on the size of the formula.  For some
formulas, we show that these upper bounds are polynomial and optimal.

\end{abstract}

\section*{Introduction}

The original motivation for developing the research that we present in
this paper was the investigation of $\mu$-calculi based on \IL.  A
$\mu$-calculus \cite{AN01} is a prototypical kind of computational
logic, obtained from a base logic or a base algebraic system, by
adding distinct forms of iteration, \lfp{s} and \gfp{s}, so to
increase expressivity.
We ended up studying \fp{s} within \IL mostly by observing structural
similarities between the propositional modal $\mu$-calculus and the
\IPC (\Ipc).  \emph{Bisimulation quantifiers} (also known as
\emph{uniform interpolants}) within the propositional modal
$\mu$-calculus were studied in \cite{DAgostinoHollenberg2000}; in this
work a formula built by using these kind of quantifiers was employed
to prove that \PDL (Propositional Dynamic Logic, see \cite{DLbook})
lacks the uniform interpolation property.  In \cite{Pitts92} the
author discovered that \Ipc also has bisimulation quantifiers;
together with the deduction property, uniform interpolants give a
rather strong structure to the category of (finitely presented) \Ha{s}
(the algebraic models of the \Ipc); this structure was axiomatized and
studied in \cite{GhilardiZawadowski2011,GhilardiZawadowski97}.
Quantified formulas analogous to the one of
\cite{DAgostinoHollenberg2000} make sense in every category with this
structure and they indeed define the \efp{s} of monotone formulas.
This made us conjecture that a $\mu$-calculus based on \IL is trivial,
meaning that every $\mu$-formula is equivalent to a \fpf formula.
The conjecture actually holds because of a deep result in \IL.
It was proved in \cite{Ruitenburg84} that, for each
formula $\phi(x)$ of the \Ipc, there exists a number $n \geq 0$ such
that $\phi^{n}(x)$---the formula obtained from $\phi$ by iterating $n$
times
substitution of $\phi$ for the variable $x$---and
$\phi^{n+2}(x)$ are provably equivalent in \IL.  An immediate
corollary of this result is that a syntactically monotone formula
$\phi(x)$ converges both to its \lfp and to its \gfp in at most $n$
steps.
We write $\mu_{x}.\phi(x) = \phi^{n}(\bot)$ and $\nu_{x}.\phi(x) =
\phi^{n}(\top)$ to express this fact, using a modern notation based on
$\mu$-calculi.
  These two identities can be used to argue that
  every formula of a
  $\mu$-calculus based on \IL is equivalent to a \fpf formula.
  
  Ruitenberg's work
  leaves open how to compute or
estimate the least number $n$ such that
$\phi^{n}(x) = \phi^{n+2}(x)$---we shall call such a number the
Ruitenberg's number of $\phi$ and denote it by $\rho(\phi)$.  As our
motivations stem from \fp theory and $\mu$-calculi, we remark that
being able to compute or bound Ruitenberg's number $\rho(\phi)$ might
yield an over-approximation of the least integer $k$ such that
$\mu_{x}.\phi(x) = \phi^{k}(\bot)$---
we call such a number $k$ 
\emph{closure ordinal of $\phi$}.
For example, when considering the dual analogous problem, and
so the \gfp of $\phi$,
we shall see that the least number $k$
such that of $\phi$ is $\nu_{x}.\phi(x) = \phi^{k}(\top)$ is $1$ at most, while
$\rho(\phi)$ can be arbitrarily large.
\Lfp{s} over \IL have also been considered in \cite{Mardaev1993}.  The
author gave there an independent proof that \lfp{s} of monotone
intuitionistic formulas are definable.  His proof relies on semantics
methods and on the coding of \IL into \GRZL; the proof was further
refined in \cite{Mardaev1994} to encompass the standard coding of \IL
into its modal companion, the logic \Sfour.

The results presented in this paper are also part of a line of
research that we are currently exploring, and that lead us to
studying \fp{s} within \IL.  We aim at identifying, under a unified
perspective, reasons that make alternation-depth hierarchies of
$\mu$-calculi degenerate or trivial.
A $\mu$-calculus is obtained by adding formal \lgfp{s} to an
underlying logical-algebraic system, so it generates \fterms with
nested \efp{s}. The alternation-depth hierarchy \cite[\S 2.6]{AN01} of
a $\mu$-calculus measures the complexity of a \fterm, as a function of
the nesting of the different types of \fp{s} and with respect to a
fixed class of models.
It is well known that \fp{s} that are unguarded can be eliminated in
the propositional modal $\mu$-calculus \cite{Kozen83}.  This fact can
be rephrased by saying that the alternation-depth hierarchy of the
$\mu$-calculus over distributive lattices is trivial (every $\mu$-term
is equivalent to a \fpf term). To closely understand and to refine
this result was one of the goals of
\cite{FrittellaSantocanaleRAMICS}. In that paper the authors were able
to exhibit equational classes of lattices $\mathcal{D}_{n}$---with
$\mathcal{D}_{0}$ the class of distributive lattices---where the
\efp{s} can be uniformly computed by iterating a \fterm $n+1$ times
from the bottom/top of the lattice; moreover, they showed that these
uniform upper bounds are optimal. For those classes of lattices, the
degeneracy of the alternation-depth hierarchy originates in the
structure of the lattices in the class.
The next and most natural algebraic setting extending distributive
lattices and where to study \fp{s}, was given by \Ha{s} and \IL.

\medskip

This paper is divided in two parts. In the first part, we firstly show
how to eliminate \gfp{s}.  Namely we argue that, for every
intuitionistic formula $\phi(x)$ with the specified variable $x$
positive in $\phi(x)$, $\nu_{x}.\phi(x) = \phi(\top)$.  \Gfp{s} of
intuitionistic formulas are reached from the top of the lattice after
one iteration, exactly as in the case of distributive lattices. At a
second stage we present the elimination procedure of \lfp{s}; the
procedure yields, for every formula $\phi(x)$ as above, a (\fpf)
intuitionistic formula $\psi$ such that $\mu_{x}.\phi(x) = \psi$.
The two elimination procedures can be casted into a procedure that
yields a \fpf formula equivalent to an arbitrary formula of the
$\muIpc$, the $\mu$-calculus based on \IL.  Since \Ipc is decidable,
the procedure also provides a decision procedure for the \muIpc.
Even if elimination of \gfp{s} turns out to be somewhat trivial, it
plays an important role for eliminating \lfp{s}. Natural properties of
\fp{s} lead to identify two orthogonal syntactic fragments of the
\Ipc: we call \emph{\weaklynegative}, resp. \emph{\fullypositive}, the
formulas belonging to these fragments.  \Lfp elimination is split
between two kind of eliminations, one for each fragment. For
\weaklynegative formulas, elimination of \lfp{s} is a consequence of
\gfp elimination. \Lfp elimination for \fullypositive formulas relies
on these formulas being inflating (i.e., semantically they give rise
to inflating monotone functions) and other ingredients.

The second part of the paper studies closure ordinals of
intuitionistic positive formulas. The closure ordinal of
$\phi(x)$---which, we recall, is the least integers $n$ for which we
can write $\mu_{x}.\phi(x) = \phi^{n}(\bot)$---yields a representation
of the \lfp $\mu_{x}.\phi(x)$ alternative to the one presented in the
first part.  Such representation can be exploited notationally, as in
$\mu$-calculi with explicit approximations \cite{damGurov},
computationally, because of its reduced space requirements, at least
if variable sharing is used, and also axiomatically.  We firstly
present general results for producing upper bounds of closure ordinals
of monotone functions and then we add results that are specific for
\Ha{s} and \IL. Whenever it is possible, we also argue that those
bounds are tight. By combing these results and, at the same time, by
paralleling the \lfp elimination procedure, upper bounds of closure
ordinals of \fterms $\phi(x)$ can be computed. It turns out that these
bounds are not tight. We focus therefore on closure ordinals of
\fullypositive \fterms that, in view of tightness of bounds, are the
most problematic. We produce specific (and better) bounds for these
formulas; in this case our proof yields bounds on Ruitenburg's numbers
and so also new insights on his theorem.  We finish the second part
of the paper 
 by presenting a syntactic fragment (formulas in
the fragment are \emph{disjunctions of} what we call \emph{\Atops})
and prove a suprising fact: closure ordinals of formulas in this
fragment have $3$ as a uniform upper bound.

\medskip

Comparing the present work to our previous results on degeneracies of
alternation-depth hierarchies, reasons for degeneracies appear now to
have a very different nature.
Several are the ingredients contributing to the existence of a finite
closure ordinal of every intuitionistic formula, thus to the
degeneracy of the alternation-depth hierarchy of the $\mu$-calculus
based on \IL.
Probably the most important among them is \emph{strongness} of
monotone polynomials on \Ha{s}. The naming comes from category theory:
a monotone polynomial $\Pf : H \rto H$ (with $H$ a \Ha) is strong if
it has a strength; in turn, this is equivalent to say that, as a
functor, it is enriched over the closed category $H$
\cite{Kock72,Kelly82}.
Yet strongness is just a possible naming for a general logical
phenomenon, the capability of an equational theory to partly encode
quasi-equations. On the proof-theoretic side, this phenomenon is known
as the deduction theorem; on the algebraic side it translates to
equationally definable principal congruences \cite{Blok1984}.  In
modal logic the deduction theorem is equivalent to having a master
modality \cite[Theorem 64]{KrachtHandbook}; as a matter of fact,
having a master modality appears to be a common pattern in several
works on alternation-depth hierarchies modal $\mu$-calculi
\cite{Mardaev1994,Mardaev2007,AlberucciFacchini09,DAgostinoLenzi2010,BS2013}.
Other ingredients are the following.
For some polynomials, existence and finiteness of the closure ordinal
is a consequence of being \emph{inflating} (or expanding) and, on the
syntactic level, to a restriction to the use of conjunction that
determines a notion of disjunctive formula.
A key ingredient of the algorithm we present is creation of \lfp{s}
via the \emph{Rolling} equation (cf. Lemma~\ref{lemma:rolling}), a
fact already used in \cite{DAgostinoLenzi2010}. For \IL and \Ha{s},
where \fterms can be semantically antitone (i.e. contravariant),
existing \gfp{s} create \lfp{s}.
Overall the most striking difference with the case of distributive
lattices and generalizations of distributive lattices
\cite{FrittellaSantocanaleRAMICS} is the absence of a finite uniform
upper bound on the closure ordinals, 
the rate of convergence to the \lfp crucially depends on the size and
shape of the formula.

\medskip

The considerations that we shall develop rely on well-known
equivalences of  \fp expressions
\cite{BloomEsik93,AN01}.
This
distinguishes our approach from previous works
\cite{Ruitenburg84,Mardaev1993}.  Using these equivalences we can move
the focus from existence and definability of \fp{s} in \IL towards the
explicit construction of them. 
On the way, let us remark that the simple characterization of \gfp{s}
in \IL $\nu_{x}.\phi(x) = \phi(\top)$, that yet plays an important
role in the elimination procedure of \lfp{s}, appears to be orthogonal
to Ruitenbutg's work, while \gfp{s} are not considered in Mardaev's
work.
The need for algorithmic approaches in \fp elimination was emphasized
in \cite{LehtinenQuickert15} for 
the propositional modal $\mu$-calculus.

\medskip

The paper is organized as follows. The goal of the first part,
Sections~Section~\ref{sec:notation} to~\ref{sec:notation}, is to
present the \fp elimination procedure for the \IPC. 
We recall in Section~\ref{sec:notation} some elementary facts from \fp
theory. In Section~\ref{sec:intmucalculus} we present the \IPC and
introduce its \fp extension, the \muIPC.  In
Section~\ref{sec:strongfunctions} we pinpoint strongness, a property
of monotone functions on \Ha{s} that will be pervasive in all the
paper. We prove some elementary facts about strong functions and their
\lfp{s} and give a simple axiomatization of their \gfp{s}. In
Section~\ref{sec:quantifiers} we digress on bisimulation quantifiers
and argue that the existence of \efp{s} can be inferred from these
quantifiers.  Section~\ref{sec:procedure} presents the elimination
procedure.

The second and last part of the paper starts with
\Section~\ref{sec:generalUpperBounds} and  deals with estimating closure
ordinals of \fterms of the \Ipc. 
We begin by presenting some general results, that apply to arbitrary
monotone functions on posets with a least element. In the second half
of \Section~\ref{sec:generalUpperBounds} we present some results
specific to \Ha{s}; the results from this \Section are sufficient to
estimate an upper bound of the closure ordinal of any \fterm, yet
these upper bounds are not tight.  Therefore we estimate in
\Section~\ref{sec:ruitenburg} closure ordinals of conjunctions of
disjunctive formulas (defined in \Section~\ref{sec:procedure}) which,
in view of tigthness of upper bounds, appear to be the most
difficult. Our work actually yields upper bounds of Ruitenburg's
numbers of these formulas and a closed expression for the formula
$\phi^{\rho(\phi)}$ (when $\phi$ is such a disjunction).
Finally, in Section~\ref{sec:almostclosure} we exemplify how the
search for bounds of closure ordinals leads to some non-trivial
discovery: we present an infinite family of \fterms that---while being
more and more complex---uniformly converge to their \lfp in $3$ steps.

\newpage
\tableofcontents
\section{Elementary \fp theory}
\label{sec:notation}

Let $P$ and $Q$ be posets. A function $f : P \rto Q$ is
\emph{monotone} if $x \leq y$ implies $f(x) \leq f(y)$, for each
$x,y \in P$.  If $f : P \rto P$ is a monotone endofunction, then
$x \in P$ is a \emph{\pfp} of $f$ if $f(x) \leq x$; we denote by
$\Pref_{f}$ the set of prefixed points of $f$. 
Whenever $\Pref_{f}$ has a least element, we denote it by $\mu.f$.
Therefore, $\mu.f$ denotes the \emph{\lpfp} of $f$, whenever it
  exists.  If $\mu.f$ exists, then it is a \fp of $f$, necessarily
the least one.  
The notions of \lpfp and of \lfp
coincide on complete lattices
or when the \lfp is computed by
iterating from the bottom of a lattice; for our purposes they are
interchangeable, so we shall abuse of language and refer to $\mu.f$ as
the \lfp of $f$.
Dually (and abusing again of language), the \emph{\gfp} of $f$ shall
be denoted by $\nu.f$.

\smallskip

Let us mention some well known identities from \fp theory, see for
example \cite{BloomEsik93} or \cite{AN01}. Notice however that the
statements that we present below also assert and emphasize the
existence of some \lfp---we do not assume completeness of the
posets. Full proofs of these statements can be found in
\cite{Santocanale:parity}.
\begin{lemma}
  \label{lemma:rolling}
  Let $P, Q$ be posets, $f : P \rto Q$ and $g : Q \rto P$ be monotone
  functions. If $\mu.(g \circ f)$ exists, then $\mu.(f \circ g)$
  exists as well and is equal to $f(\mu.(g \circ f))$. 
\end{lemma}
As we do not work in complete lattices (so we are not ensured that
\lfp{s} exist) we express the above statement via the equality
\begin{align}
  \label{eq:rolling}
  \tag{\texttt{Roll}}
  \mu.(f \circ g) & := f(\mu.(g \circ f))\,,
\end{align}
where the colon emphasizes existence: if the \lfp in the expression on
the right \emph{exists}, then this expression is the \lfp of
$f\circ g$. Analogous notations will be used later. We endow the
  product of two posets $P$ and $Q$ with the coordinatewise ordering.
  Therefore a function $f : P \times Q \rto R$ is monotone if,
  as a
  function of two variables,
   it is monotone in each variable. To deal
with \lfp{s} of functions of many variables, we use the standard
notation: for example, if $f : P\times P\rto P$ is the monotone
function $f(x,y)$, then, for a fixed $p \in P$, $\mu_{x}.f(x,p)$
denotes the \lfp of $f(x,p)$. Let us recall that the correspondence
$p \mapsto \mu_{x}.f(x,p)$---noted $\mu_{x}.f(x,y)$---is again
monotone.
\begin{lemma}
  \label{lemma:diag}
  If $P$ is a poset and $f : P\times P \rto P$ is a monotone mapping,
  then
  \begin{align}
    \label{eq:diag}
    \tag{\texttt{Diag}}
    \mu_{x}.f(x,x) & := \mu_{x}.\mu_{y}.f(x,y)\,.
  \end{align}
\end{lemma}
Again, the expression above shall be read by saying that if
$\mu_{y}.f(x,y)$ exists, for each $x \in P$, and if
$\mu_{x}.\mu_{y}.f(x,y)$ exists, then $\mu_{x}.f(x,x)$ exists as well
and is given by the expression on the right.  
  
Recall that a function $f$ from $A$ to a product $B \times C$ is
uniquely determined by two functions $g : A \rto B$ and $h : A \rto C$
via composing with projections; we therefore write
$f = \langle g,h\rangle$ and use a similar notation for products with
more factors.
\begin{lemma}
  \label{lemma:bekic}
  If $P$ and $Q$ are posets and $\langle f,g\rangle : P\times Q \rto
  P\times Q$ is a monotone function, then $\mu.\langle f,g \rangle :=
  \langle \mu_{1},\mu_{2}\rangle$, where
  \begin{align}
    \label{eq:bekic}
    \tag{\texttt{\Bekic}}
    \mu_{1} & = \mu_{x}.f(x,\mu_{y}.g(x,y)) \quad \tand\quad \mu_{2} =
    \mu_{y}.g(\mu_{1},y)\,.
  \end{align} 
\end{lemma}

\section{The  \muIPC}
\label{sec:intmucalculus}

Formulas of the \IPC (\Ipc) are generated according to the following grammar:
\begin{align}
  \label{grammar:IL}
  \phi & \production x \mid \top \mid \phi \land \phi \mid \bot \mid \phi
  \vee \phi \mid \phi \impl \phi\,, 
\end{align}
where $x$ ranges over a countable set $\Vars$ of propositional
variables.
The set of these formulas shall be denoted $\FIpc$.
The consequence relation of the \Ipc, relating a set of formulas to a
formula, is described by means of the intuitionisitc sequent calculus,
Gentzen's system \LJ \cite{Gentzen35}.
Therefore we shall write $\Gamma \cons \delta$ if the
sequent $\Gamma \vdash \delta$ is derivable in the system $\LJ$.

\medskip

It is well known that the \Ipc is sound and complete w.r.t. the class
of its algebraic models, the \Ha{s} that we introduce next.
\begin{definition}
  A \emph{\Ha} $H$ is a bounded lattice (so $H$ has a least element
  $\bot$ and a greatest element $\top$) equipped with a binary
  operation $\impl$ such that the following equations hold in $H$:
  \begin{align}
    \notag
    x \land (x \impl y) & = x \land y\,, &
    x \land (y \impl x) & = x\,,\\
    x \impl x & = \top \,, &
    x \impl (y \land z) & = (x \impl y) \land  (x \impl z)\,.
    \label{eq:distrimpl}
  \end{align}
\end{definition}
We can define on any \Ha a partial order by saying that $x \leq y$
holds when $x \vee y = y$.  We identify formulas of the \Ipc with
terms of the theory of \Ha{s}, constructed therefore from variables
and using the signature $\langle \top,\land, \bot,\vee, \impl\rangle$;
we shall therefore refer to objects generate by the grammar
\eqref{grammar:IL} as \emph{\fterms}.\footnote{
  In view of the verbosity of the naming \fterms we shall often use
  formula or term as a synonym of \fterm. }
Let $\phi$ 
be such a \fterm, let $H$ be a \Ha and let $v : \Vars \rto H$
be a valuation of the propositional variables in $H$; we write
$\eval{\phi}$ for the result of evaluating the formula in $H$,
starting from the variables (the definition of $\eval{\cdot}$ is given
as usual by induction).
The soundness and completeness theorem of the \Ipc over \Ha{s}, see
e.g. \cite{BJ2006}, can then be stated as follows:
\begin{thm*}
  If  $\,\Gamma$ is a finite set of \fterms and $\phi$ is a \fterm, then
  $\Gamma \cons \phi$ holds if and only if
  $\bigwedge_{\gamma \in \Gamma} \eval{\gamma} \leq \eval{\phi}$
  holds, in every \Ha $H$ and for every valuation of the propositional
  variables $v : \Vars \rto H$.
\end{thm*}
Given this theorem, we shall often abuse of notation and write $\leq$
in place of $\cons$, and the symbol $\eqIpc$ (or even the equality
symbol $=$) to denote provable equivalence of formulas.  That is to
say, we identify \fterms with elements of the \LTa of the \Ipc.
Recall that this algebra is also the free \Ha over the set $\Vars$ of
propositional variables; therefore we shall denote it by
$\FH[\Vars]$. More generally, for a set of generators $Y$, the free
\Ha on this set shall be denoted by $\FH[Y]$.

\medskip

We aim at studying \efp{s} on \Ha{s} and so we formalize next the
\muIPC (\muIpc).

An occurrence of a variable $x$ is \emph{positive} in a \fterm $\phi$
if, in the syntax tree of $\phi$, the path from the root to the leaf
labeled by this variable occurrence contains an even number of nodes
such that: (i) they are labeled by a subformula
$\psi_{1} \impl \psi_{2}$ and (ii) their immediate successor on the
path is labeled by the subformula $\psi_{1}$.  If on this path the
number of those nodes is odd, then we say that this occurrence of $x$
is \emph{negative} in $\phi$.  
For example, in the formula-term $((x\impl y)\impl (x\lor z))\impl w$
the first occurrence of $x$ is positive but the second occurrence is
negative.
A variable $x$ is positive in a formula
$\phi$ if each occurrence of $x$ is positive in $\phi$. A variable $x$
is negative in a formula $\phi$ if each occurrence of $x$ is negative
in $\phi$.
We enrich the grammar~\eqref{grammar:IL} with the following two
productions:
\begin{align*}
  \phi \production & \mu_{x}.\phi\,,
  & \phi &\production \nu_{x}.\phi\,,
\end{align*}
subject to the restriction that $x$ is positive in $\phi$; we obtain
in this way a grammar for formulas of the \muIpc. The set of formulas
generated by this grammar shall be denoted by $\FmuIpc$.
Notice that the symbols $\mu$ and $\nu$
syntactically behaves as binders (similar to quantifiers), so the
notions of free and bound variable in a \fterm is defined as usual.

\medskip
 
We present next the semantics
of the \muIpc over \Ha{s}. An equivalent
formulation of the \muIpc, via a sequent calculus, appears in \cite[\S
2]{Clairambault13}.

For a fomula $\phi$ of the $\muIpc$, let $\Vars[\phi]$ denote the set
of variables having a free occurrence in $\phi$. Let $H$ be a \Ha
(that we do not suppose complete); we define next a \emph{partial}
evaluation function sending $\phi \in \muIpc$ and
$v : \Vars[\phi] \rto H$ to $\eval{\phi}$.  We only cover the cases of
formulas $\mu_{x}.\phi$ and $\nu_{x}.\phi$,
since the other cases are the usual ones.
Thus let $\phi$ be a formula of the \muIpc, let $x$ be positive in
$\phi$, and suppose that $\eval[u]{\phi}$ is defined, for each
$u : \Vars[\phi] \rto H$.\footnote{If, for some
  $u : \Vars[\phi] \rto H$, $\eval[u]{\phi}$ is not defined, then
  $\eval{\mu_{x}.\phi},\eval{\nu_{x}.\phi}$ are not defined.} 
If $v : \Vars[\phi] \setminus \set{x}\rto H$ is a valuation of all the
free variables of $\phi$ but $x$, then $(v,h/x) : \Vars[\phi] \rto H$
is the valuation such that $(v,h/x)(x) = h$ and $(v,h/x)(y) = v(y)$
for $y \neq x$.  Since $x$ is positive in $\phi$, then the function
\begin{align*}
  \eval[v]{\phi} & : H \rto H \,,
  \quad
  h  \mapsto \eval[(v,h/x)]{\phi}\,
\end{align*}
is monotone; therefore, if the \efp{s}
of this function exist,\footnote{If any 
of the \efp{s} does not exist,
  then we leave the corresponding expressions undefined.} then we
define
\begin{align*}
  \eval{\mu_{x}.\phi}
  & \eqdef \Mu{\eval{\phi}}, 
  &
  \eval{\nu_{x}.\phi}
  & \eqdef  \Nu{\eval{\phi}}\,. 
\end{align*}
Clearly, when $H$ is a complete \Ha, then \efp{s} of monotone
functions exists, so the correspondence $(\phi,v) \mapsto \eval{\phi}$
is total. We argue next that this correspondence is \emph{always}
total.

We say that two formulas $\phi$ and $\psi$ of the $\muIpc$ are
equivalent over \Ha{s} if, for each \Ha $H$ and each
$v : \Vars[\phi] \cup \Vars[\psi] \rto H$, $\eval{\phi}$ is defined if
and only if $\eval{\psi}$ is defined, and $\eval{\phi} = \eval{\psi}$
whenever they are both defined. We write $\phi \eqmuIpc \psi$ when two
formulas $\phi$ and $\psi$ of the $\muIpc$ are equivalent over \Ha{s}.

\medskip

Let us say that a formula $\phi$ of $\muIpc$ is \emph{\fpf} if it does
not contain either of the symbols $\mu,\nu$ (that is, if it is a
formula of the $\Ipc$).
\begin{proposition}
  \label{prop:total}
  Every formula $\phi$ of the $\muIpc$ is equivalent over \Ha{s} to a
  \fpf formula $\psi$.  In particular $\eval{\phi}$ is defined, for
  each \fterm $\phi$ of the $\muIpc$, each \Ha $H$, and each
  $v : \Vars[\phi] \rto H$,
\end{proposition}
\begin{proof}
  Clearly, the first statement of the Proposition holds if we can show
  that it holds whenever $\phi = \mu_{x}.\psi$ or
  $\phi = \nu_{x}.\psi$, where $\psi$ is a \fpf formula.
  For a natural number $n \geq 0$, let $\psi^{n}(x)$
  denote the formula obtained by substituting $x$ for $\psi$ $n$
  times. 
  Ruitenburg \cite{Ruitenburg84} proves that, for each intuitionisitic
  propositional formula $\psi$, there exists a number $n \geq 0$ such
  that the formulas $\psi^{n}(x) \eqIpc \psi^{n + 2}(x)$.
  If $x$ is positive in $\psi$, then instantiating $x$ with $\bot$,
  leads to the equivalence $\psi^{n+1}(\bot) \eqIpc
  \psi^{n}(\bot)$. Yet this relation enforces $\psi^{n}(\bot)$ to be
  the \lfp of $\psi$, namely
  $\eval{\mu_{x}.\psi} = \eval{\psi^{n}(\bot)}$ for each $H$ and
  $v : \Vars[\psi] \rto H$. That is, we have
  $\mu_{x}.\psi \eqmuIpc \psi^{n}(\bot)$; similarly, we get
  $\nu_{x}.\psi \eqmuIpc \psi^{n}(\top)$.
\end{proof}
According to the Proposition (and to Ruitenburg's result
\cite{Ruitenburg84}) the expansion of the \Ipc with \efp{s} does not
increase its expressive power. This does not exclude the use of
$\muIpc$ as a convenient formalism, but raises the problem of
(efficiently) computing, for each $\phi \in \FmuIpc$, a formula
$\psi \in \FIpc$ such that $\phi \eqmuIpc \psi$.

For a formula $\mu_{x}.\phi$ with $\phi$ \fpf, this can be achieved by
computing the Ruitenburg's numbers $\rho(\phi)$. An attentive reading
of Ruitenburg's paper shows that $\rho(\phi) \leq 2n + 2$ where $n$ is
the size of the formula. Yet, $\rho(\phi)$ might not be an optimal as
an upper bound to $n$ such that $\mu_{x}.\phi \eqmuIpc \phi^{n}(\bot)$
or $\nu_{x}.\phi \eqmuIpc \phi^{n}(\top)$.

\section{Strong monotone functions and \fp{s}}
\label{sec:strongfunctions}

If $H$ is a Heyting algebra and $f : H \rto H$ is any function, then
$f$ is said to be \emph{compatible} if
\begin{align}
  x \land f(y) & = x \land f(x \land y)\,,
  \quad\text{for any $x, y \in H$.}
   \label{eq:compatible}
\end{align}
\begin{remark}
  We are mainly interested in monotone functions. If $f$ as above is
  also \emph{monotone}, then $f$ is compatible if and only if it is
  \emph{strong}, meaning that it satisfies
  \begin{align}
    x \land f(y) & \leq f(x \land y)\,, &&\text{for any $x, y \in H$.} \label{eq:strongconj} 
  \end{align}
  The interplay between \fp{s} and strong monotone functions has
  already been emphasized, mainly in the context of categorical
  proof-theory and semantics of functional programming languages with
  inductive data types \cite{CockettSpencer95,Clairambault13}.  It is
  well known from categorical literature \cite{Kock72} that
  condition~\eqref{eq:strongconj} is equivalent to any of the
  following two conditions:
  \begin{align}
    f(x \impl y) & \leq x \impl f(y)\,, &&\text{for any $x, y \in H$,}\label{eq:strongimpl} \\
    x \impl y & \leq f(x) \impl f(y)\,,  &&\text{for any
      $x, y \in H$.}\label{eq:strongenriched}
  \end{align}
\end{remark}

\medskip

Recall that if $v : \Vars[\phi] \setminus \set{x}\rto H$ is a
valuation of all the free variables of $\phi$ but $x$, then
$(v,h/x) : \Vars[\phi] \rto H$ is the valuation such that
$(v,h/x)(x) = h$ and $(v,h/x)(y) = v(y)$ for $y \neq x$.
\begin{definition}
  \label{def:monotonepol}
  Let $H$ be a Heyting algebra. We say that a function $\Pf : H \rto H$
  is a \emph{polynomial} if there exist a formula $\phi \in\FIpc$, a
  variable $x$, and a valuation
  $v : \Vars[\phi] \setminus \set{x} \rto H$ such that, for each
  $h \in H$, we have $\Pf(h) = \eval[(v,h/x)]{\phi}$.
\end{definition}
Equivalently, a polynomial on $H$ can be identified with an element of
the polynomial \Ha $H[x]$, where the last is defined as the coproduct
(in the category of \Ha{s}) of $H$ with the free \Ha on one
generator. In Section~\ref{sec:quantifiers} we shall study further
such polynomial algebras and exploit their properties.

In the next Propositon, the analogous statement for Boolean algebras
is credited to Peirce, in view of the iteration rule for existential
graphs of type Alpha \cite{Dau2006}.
\begin{proposition}
  \label{prop:Peirce}
  Every polynomial $\Pf$ on a \Ha is compatible. In particular, if $\Pf$
  is monotone, then it is strong.
\end{proposition}
\begin{longversion}
  \begin{proof}
    If $H$ is an Heyting algebra, then an upset $F \subseteq H$ is a
    filter if it is closed under meets. We recall that filters $F$
    bijectively correspond to congruences $\theta$, by saying $x\theta
    y$ if and only if $x \leftrightarrow y \in F$.

    Thus, if $\Pf : H \rto H$ is a polynomial, then write it as
    $\eval[]{\phi}(z,\vec{v})$ for a tuple of elements of $\vec{v}$ of
    $H$.

    Let also $H' = H/F$, where $F$ is the upset of $x \impl y$, so
    $\eval[]{\phi}(z,[\vec{v}])$ is a monotone polynomial on $H'$. As
    in $H'$ we have $[x] \leq [y]$, then
    \begin{align*}
      [\Pf(x)] = [\,\eval[]{\phi}(x,\vec{v})\,] & =
      \eval[]{\phi}([x],[\vec{v}]) \\
      & \leq \eval[]{\phi}([y],[\vec{v}])
      \\
      & = [\,\eval[]{\phi}(y,\vec{v})\,] = [\Pf(y)].
    \end{align*}
    Yet, the inclusion $[\Pf(x)] \leq [\Pf(y)]$ in $H'$ means that
    $\Pf(x) \impl \Pf(y)$ is in the principal filter generated by $x
    \impl y$, that is, that equation \eqref{eq:strongenriched} from
    Lemma~\ref{lemma:strong} holds.  
  \end{proof}
\end{longversion}
\begin{proof}
  Recall that the replacement Lemma holds in the \Ipc:
  $z \leftrightarrow w \cons \phi(z) \biimpl \phi(w)$.  Substituting
  $y$ for $z$ and $x \land y$ for $w$, and considering that
  $x \cons y \biimpl (x \land y) $, we derive that
  $x \cons \phi(y) \biimpl \phi(x \land y)$.  The latter relation is
  equivalent to the conjunction of
  $x \land \phi(y) \cons x \land \phi(x \land y)$ and
  $x \land \phi(x \land y) \cons x \land \phi(y)$. These two
  relations immediately imply that equation~\eqref{eq:compatible}
holds  when
  $\Pf$ is a polynomial.
\end{proof}

On the way let us include the following Lemma.
\begin{lemma}
  \label{lemma:idempotency}
  If $\Pf : H \rto H$ is a strong monotone function and $a \in H$,
  then
  \begin{align}
    \label{eq:idempotency}    
    a \impl \Pf(a \impl x)
    & = a \impl \Pf(x)\,.
  \end{align}
\end{lemma}
\begin{proof}
  Using \eqref{eq:strongimpl}, we deduce
  \begin{align*}
    a \impl \Pf(a \impl x) & \leq a \impl (a \impl
    \Pf(x))  = a \impl \Pf(x)\,.
  \end{align*}
  The converse  relation follows from $x \leq a \impl x$ and
  $a \impl \Pf(x)$ being monotone in $x$.
\end{proof}

\begin{proposition}
  \label{prop:fixrimpl}
  If $\Pf$ is a strong monotone function on $H$ and $a \in H$, then
  \begin{align}
    \label{eq:fixrimpl-conj}
    \MuP{a \impl \Pf} & := a \impl \Mu{\Pf}\,,
    &
    \MuP{a \land \Pf} & := a \land \Mu{\Pf}\,.
  \end{align}
\end{proposition}
\begin{proof}
  Firsty, we argue that the equation on the left holds. To this end,
  let us set $\fea(x) \eqdef a \impl \Pf(x)$.  From $\Pf \leq \fea$ we
  have $\Pref_{\fea} \subseteq \Pref_{\Pf}$. Thus, if
  $p \in \Pref_{\fea}$, then $\Mu{\Pf} = \Pf(\Mu{\Pf}) \leq \Pf(p)$ and
  $a \impl\Mu{\Pf \leq a \impl \Pf(p)} = \fea(p) \leq p$. That is,
  $a \impl\Mu{\Pf} $ is below any element of $\Pref_{\fea}$. To obtain
  the proposition, we need to argue that $a \impl\Mu{\Pf }$ belongs to
  $\Pref_{\fea}$.  This follows from equation~\eqref{eq:idempotency}:
  $\fea(a \impl\Mu{\Pf }) = a \impl \Pf(a \impl \Mu{\Pf }) = a \impl
  \Pf(\Mu{\Pf }) = a \impl \Mu{\Pf }$.

  \smallskip 

  Let us come now to the equation on the right, for which we set
  $\fla(x) \eqdef a \land \Pf(x)$.  Suppose $a \land \Pf(p) \leq p$, so
  $\Pf(p) \leq a \impl p$. Then
  $\Pf(a \impl p) \leq a \impl \Pf(p) \leq a \impl p$, using
  \eqref{eq:strongimpl}, whence $\Mu{\Pf} \leq a \impl p$ and
  $a \land\Mu{\Pf} \leq p$.
  Thus we are left to argue that $a \land \Mu{\Pf}$ is a \pfp of
  $\fla$. Yet, this is true for an arbitrary \pfp $p$ of $\Pf$: $a
  \land \Pf(a \land p) \leq a \land \Pf(p) \leq a \land p$. 
\end{proof}

\begin{corollary}
  \label{cor:distrconj}
  For each $n \geq 1$ and each collection $\Pf_{i}$, $i = 1,\ldots ,n$
  of monotone polynomials, we have the following distribution law:
  \begin{align}
    \mu_{x}. \bigwedge_{i = 1,\ldots ,n} \Pf_{i}(x) & := \bigwedge_{i =
      1,\ldots ,n} \mu_{x}.\Pf_{i}(x)\,.
    \label{eq:muconjunction}
  \end{align}
\end{corollary}
\begin{proof}
  For $n = 1$ there is nothing to prove. We suppose therefore that the
  statement holds for every collection of size $n \geq 1$, and prove it holds
  for a collection of size $n+1$. We have
  \begin{align*}
    \mu_{x}.(\Pf_{n+1}(x) \land \bigwedge_{i=1,\ldots ,n} \Pf_{i}(x)) &
    := \mu_{x}.\mu_{y}.(\Pf_{n+1}(y) \land \bigwedge_{i=1,\ldots
      ,n} \Pf_{i}(x)), \tag*{by \eqref{eq:diag},}\\
      & := \mu_{x}.( (\mu_{y}.\Pf_{n+1}(y)) \land \bigwedge_{i=1,\ldots
        ,n} \Pf_{i}(x)), \tag*{by~\eqref{eq:fixrimpl-conj},}\\
      & := (\mu_{y}.\Pf_{n+1}(y)) \land \mu_{x}.(\bigwedge_{i=1,\ldots
        ,n} \Pf_{i}(x)), \tag*{again by~\eqref{eq:fixrimpl-conj},}\\
      & := (\mu_{y}.\Pf_{n+1}(y)) \land \bigwedge_{i=1,\ldots ,n}
      \mu_{x}.\Pf_{i}(x), \tag*{by the IH. \quad \qedhere}
  \end{align*}
\end{proof}

The elimination of \gfp{s} is easy for strong monotone functions. We
are thankful to a referee of \cite{fossacs} for pointing out the
following fact, which greatly simplified our original argument:
\begin{proposition}
  \label{prop:phitop}
  If $\Pf : L \rto L$ is any strong monotone function on a bounded
  lattice $L$, then $\Pf^{2}(\top) = \Pf(\top)$. Thus $\Pf(\top)$ is the
  \gfp of $\Pf$.
\end{proposition}
\begin{proof} Indeed, we have
$\Pf(\top) = \Pf(\top) \land \Pf(\top) \leq \Pf(\Pf(\top) \land \top) =
\Pf^{2}(\top)$.
\end{proof}

\section{Bisimulation quantifiers and \fp{s}}
\label{sec:quantifiers}
\label{sec:gfp}

The connection between \efp{s} and bisimulation quantifiers, firstly
emphasized in \cite{DAgostinoHollenberg2000}, was a main motivation to
develop this research. Although in the end the elimination procedure
does not rely on 
it, we nevertheless want to have a
closer look at this connection. 
It was discovered in \cite{Pitts92} that \Ipc has the uniform
interpolation property. 
As it is clear from the title of that work, 
this property amounts to an internal existential and universal
quantification. This result was further refined in
\cite{GhilardiZawadowski97} to show that any morphism between finitely
presented \Ha{s} has a left and a right adjoint.

We shall be interested in \Ha{s} $H[x]$ of polynomials with
coefficients 
in $H$, and 
in particular 
mappings
from $H[x]$ to $H$, namely the left and
right adjoints to the
inclusion of $H$ into $H[x]$. The algebra of polynomials $H[x]$ is
formally defined as the coproduct (in the category of \Ha{s}) of $H$
with the free \Ha on one generator. The universal property of the
coproduct yields that \emph{for every $h_{0} \in H$ there exists a
  unique morphism $\eval[h_{0}/x]{\cdot} : H[x] \rto H$ such that
  $\eval[h_{0}/x]{x} = h_{0}$ and $\eval[h_{0}/x]{h} = h$, for each
  $h \in H$}. Thus, for $\Pf \in H[x]$ and $h \in H$, we can define
the result of evaluating $f$ at $h$ by
$\Pf(h) \eqdef \eval[h/x]{\Pf}$.
If $H$ is finitely generated, then the correspondence sending $h$ to
$\Pf(h) = \eval[h/x]{\Pf}$ is a polynomial on $H$, as defined in
Definition~\ref{def:monotonepol}; moreover, every polynomial in $H$ 
arises from some $\Pf \in H[x]$ in this way.

It was proved in \cite{GhilardiZawadowski97} that if $H$ is finitely
presented, then the canonical inclusion $i_{x} : H \rto H[x]$ has both
adjoints $\exists_{x} ,\forall_{x} : H[x] \rto H$, with $\exists_{x}$
is left adjoint to $i_{x}$ and $\forall_{x}$ is right adjoint to
$i_{x}$. This means that, for each $f \in H[x]$ and $h \in H$, the
following equivalences hold:
\begin{align}
  \exists_{x}.f \leq h & \quad\tiff\quad f \leq i_{x}(h)\,, &
  h \leq \forall_{x}.f & \quad\tiff\quad i_{x}(h) \leq f\,.
  \label{eq:adjoints}
\end{align}
From these relations the unit relation for $\exists_{x}$ and the
counit relation for $\forall_{x}$ are easily derived:
\begin{align}
  \label{eq:unitPol}
  \Pf & \leq i_{x}(\exists_{x}.\Pf) \,,
  &
  i_{x}(\forall_{x}.\Pf) & \leq \Pf\,,
  \qquad \text{for all $\Pf \in H[x]$}\,.
\end{align}
We shall use in the rest of this section a standard informal notation:
we write $f(x)$ for $f \in H[x]$ and identify $h \in H$ with the
constant polynomial $i_{x}(h) \in H[x]$. Using these conventions, the
inequalities in \eqref{eq:unitPol} are written respectively as
$\Pf(x) \leq \exists_{x}.\Pf(x)$ and $\forall_{x}.\Pf(x) \leq \Pf(x)$.
We say that $\Pf \in H[x]$ is monotone if the evaluation function it
gives rise is monotone, that is, if
$\eval[h_{0}/x]{\Pf} \leq \eval[h_{1}/x]{\Pf}$ whenever
$h_{0} \leq h_{1}$.
\begin{proposition}
  \label{prop:gfpFromEx}
  If $\Pf$ is a monotone polynomial on a finitely presented Heyting
  algebra, then
  \begin{align}
    \label{eq:nuexistsone}
    \nu.\Pf & := \exists_{x}.( x \land (x \impl \Pf(x)))\,.
  \end{align}
\end{proposition}
\begin{proof}
  By the unit relation in \eqref{eq:unitPol}
  $ x \land (x \impl \Pf(x)) \leq \exists_{x}.(x \land (x \impl \Pf(x)))$.
  Recall that evaluation at $p \in H$ is a \Ha morphism, thus it is
  monotone. Therefore, if $p \in H$ is a \pofp of $\Pf$, then by
  evaluating the previous inequality at $p$, we have
  \begin{align*}
    p & = p\land  (p \impl \Pf(p)) \leq  \exists_{x}.(x \land (x \impl
    \Pf(x)))\,,
  \end{align*}
  so that $\exists_{x}.(x \land (x \impl \Pf(x)))$ is greater than any
  \pofp of $\Pf$. Let us show that
  $\exists_{x}.(x \land (x \impl \Pf(x)))$ is also a \pofp. In view of
  \eqref{eq:adjoints} it will be enough to argue that
  $x \land x \impl \Pf(x) \leq \Pf(\exists_{x}.(x \land x \impl
  \Pf(x)))$ in $H[x]$. We compute as follows:
  \begin{align*}
    x \land
  (x \impl \Pf(x))
  & \leq \Pf(x) \land (x \impl \Pf(x)) \\
  & \leq \Pf(x \land (x \impl \Pf(x))), \tag*{since $\Pf$ is strong, by \eqref{eq:strongconj},} \\
  & \leq \Pf(\exists_{x}.(x \land (x \impl \Pf(x)))), \tag*{since
    $\Pf$ is monotone. \qedhere} 
  \end{align*}
\end{proof}

In a similar fashion, we can construct \lfp{s} of monotone polynomials
using this time \emph{universal} bisimulation quantifiers.
\begin{proposition}
  \label{prop:lfpFromAll}
  If $\Pf$ is a monotone polynomial on a finitely presented Heyting
  algebra, then
  \begin{align*}
    \mu.\Pf & := \forall_{x}.((\Pf(x) \impl x) \impl x)\,.
  \end{align*}
\end{proposition}
\begin{proof}
  By the counit relation in \eqref{eq:unitPol}
  $\forall_{x}.((\Pf(x) \impl x) \impl x) \leq (\Pf(x) \impl x) \impl
  x$.  
  Evaluating this relation at $p \in H$ such that
  $\Pf(p) \leq p$, we obtain
  \begin{align*}
    \forall_{x}.((\Pf(x) \impl x) \impl x)
    & \leq (\Pf(p) \impl p) \impl p = \top \impl p = p\,,
  \end{align*}
  so  $\forall_{x}.((\Pf(x) \impl x) \impl x)$ is smaller than any
  \pfp of $\Pf$.  We show next that
  $\forall_{x}.((\Pf(x) \impl x) \impl x)$ is also a \pfp of $f$ for
  which it will be enough to argue that
  $\Pf(\,\forall_{x}.((\Pf(x) \impl x) \impl x)\,) \leq (\Pf(x) \impl x)
  \impl x$ in $H[x]$ or, equivalently, that
  $(\Pf(x) \impl x) \land \Pf(\forall_{x}.((\Pf(x) \impl x) \impl x))
  \leq x$.
  We
  compute as follows:
  \ppfootnotesize
  \begin{align*}
    (\Pf(x)  \impl x) \land \Pf(\,\forall_{x}.((\Pf(x) \impl x) \impl x)\,)
    & =
    (\Pf(x)  \impl x) \land \Pf(\,(\Pf(x)  \impl x) \land
    \forall_{x}.((\Pf(x) \impl x) \impl x)\,)\,,
    \tag*{where we use that $\Pf$ is strong,}
    \\
    & \leq
    (\Pf(x) \impl x) \land \Pf(x) \leq x 
  \end{align*}
  \ppnormal
where in the last inequality we have used that $f$ is monotone and the
relation
$(f(x) \impl x) \land \forall_{x}.((\Pf(x) \impl x) \impl x) \leq x$,
equivalent to the counit relation for $(\Pf(x) \impl x) \impl x$.
\end{proof}

The next result is an immediate consequence of
Propositions~\ref{prop:Peirce} and \ref{prop:phitop}. However the
previous proposition
yields now an alternative proof:
\begin{corollary}
  \label{cor:phitop}
  If $\Pf$ is a monotone polynomial on a \Ha $H$, then
  \begin{align}
    \label{eq:nuexists}
    \nu.\Pf & := \Pf(\top)\,.
  \end{align}
\end{corollary}
\begin{proof}
  It is easy to see that if $\Pf$ is a monotone polynomial on a
  finitely presented \Ha, then $\exists_{x}.\Pf = \Pf(\top)$. Thus we
  have
  \begin{align*}
    \nu.\Pf & = \exists_{x}.( x \land (x \impl \Pf(x))) = \exists_{x}.(
    x \land \Pf(x)) = \top \land \Pf(\top) = \Pf(\top)\,. \tag*{\qedhere}
  \end{align*}
\end{proof}
Let us come back to a more syntactic perspective. If $\phi(x)$ is a
\fterm positive in $x$ whose variables distinct from $x$ are among
$y_{1},\ldots ,y_{n}$, then the equality $\phi^{2}(\top) = \phi(\top)$
holds in the free \Ha on the set $\set{y_{1},\ldots ,y_{n}}$ (which is
finitely presented).  Since such a free \Ha is a subalgebra of the
\LTa, this means that $\phi(\top) \cons \phi(\phi(\top))$ and
$\phi(\phi(\top)) \cons \phi(\top)$.

\section{The elimination procedure}
\label{sec:procedure}

We present in this Section our first main result, a procedure that
both axiomatizes and eliminates \lfp{s} of the form $\mu_{x}.\phi$
with $\phi$ \fpf. Together with the axiomatization of \gfp{s} given in
Proposition~\ref{prop:phitop} and Corollary~\ref{cor:phitop}, the
procedure can be extended to a procedure to construct a \fpf formula
$\psi$ equivalent to a given formula $\chi$ of the \muIpc.  
To ease the reading of the content of this Section and of the remaing
ones, we introduce the following notation:
\begin{align*}
  \Nec[\alpha] \phi & \eqdef \alpha \impl \phi\,.
\end{align*}
When using the notation above, we shall always assume that the
special
variable $x$ does not occur in the formula $\alpha$.

\begin{definition}
  An occurrence of the variable $x$ is \emph{\stronglypositive} in a
  \fterm $\phi$ if there is no subformula $\psi$ of $\phi$ of the form
  $\psi_{0} \impl \psi_{1}$ such that $x$ is located in $\psi_{0}$. A
  \fterm $\phi$ is \emph{\stronglypositive} in the variable $x$ if
  every occurrence of $x$ is \emph{\stronglypositive} in $\phi$.
  An occurrence of a variable $x$ is \emph{\weaklynegative} in a
  \fterm $\phi$ if it is not \stronglypositive. A \fterm $\phi$ is
  \emph{\weaklynegative} in the variable $x$ if every occurrence of
  $x$ is \emph{\weaklynegative} in $\phi$.
\end{definition}
We shall also say that a variable $x$ is \stronglypositive
(resp. \weaklynegative) in a formula $\phi$ when $\phi$ is
\stronglypositive (resp. \weaklynegative) in the variable $x$.
Observe that a variable might be neither \stronglypositive nor
\weaklynegative in a \fterm.

\subsection{Summary of the procedure}
In order to compute the \lfp $\mu_{x}.\phi$, we 
take the following steps:
\begin{enumerate}
\item We rename all the \wnegative occurrences of $x$ in $\phi$ to a
  fresh variable $y$, so $\phi(x) = \psi(x,x/y)$ with $\psi$
  \stronglypositive in $x$ and \weaklynegative in $y$.
\item \emph{Computation of a normal form.} We compute a normal form of
  $\psi(x,y)$, that is, a formula equivalent to $\psi(x,y)$ which is a
  conjunction $\bigwedge_{i \in I} \psi_{i}(x,y)$ with each $\psi_{i}$
  disjunctive in $x$ (see Definition~\ref{def:disjunctive} below) 
   or not containing the variable $x$.
\item \emph{\Spositive elimination.} For each $i \in I$:
  if $x$ has an occurrence in $\psi_{i}$, we compute then a formula
  $\psi'_{i}$ equivalent to the \lfp $\mu_{x}.\psi_{i}(x,y)$ and
  observe that $\psi'_{i}$ is \weaklynegative in $y$; otherwise, we
  let $\psi'_{i} = \psi_{i}$.
  \item \emph{\Wnegative elimination}. The formula $\bigwedge_{i \in
      I} \psi'_{i}(y)$ is \wnegative in $y$; we compute a formula
    $\chi$ equivalent to $\mu_{y}.\bigwedge_{i} \psi'_{i}(y)$ and
    return it.
\end{enumerate}
The correction of the procedure relies on the following chain of
equivalences:
\begin{align*}
  \mu_{x}.\phi(x) & = \mu_{y}.\mu_{x}.\psi(x,y)\,,  \tag*{where we use \diag,}\\
  & = \mu_{y}.\mu_{x}.\bigwedge_{i \in I}\psi_{i}(x,y)
  =
  \mu_{y}.\bigwedge_{i \in I}\mu_{x}.\psi_{i}(x,y) \,,
  \tag*{using Corollary~ \ref{cor:distrconj},}
  \\& 
  = \mu_{y}.\bigwedge_{i \in I}\psi'_{i}(y) = \chi\,.
\end{align*}

\subsection{Computation of a normal form}
If a \fterm $\phi$ does not contain the variable $x$, then $x$ is both
\stronglypositive and \weaklynegative in this formula. Yet, in this
case, we have $\mu_{x}.\phi = \phi$, thus it is a trivial case for the
sake of computing its \lfp.
For this reason  we present below a grammar recognising \fullypositive
\fterms containing the variable $x$. The grammar is 
\begin{align}
  \phi & \production x \mid \Nec[\alpha] \phi \mid \beta \vee \phi
  \mid \phi \vee \phi \mid \phi \land \phi \mid \gamma \land \phi\,
    \label{grammar:fullypositive}
\end{align}
where conjunctions and disjunctions are taken up to commutativity and
where $\alpha,\beta,\gamma$ do not contain the variable $x$.  Another
key concept for the elimination procedure is the notion of disjunctive
formula, obtained by eliminating the last two productions from the
above grammar.
\begin{definition}
  \label{def:disjunctive}
  The set of \fterms that are \emph{disjunctive in the variable $x$}
  is generated by the grammar
  \begin{align}
    \phi & \production x \mid \Nec[\alpha] \phi \mid \beta \vee \phi
    \mid  \phi \vee \phi\,,
    \label{grammar:disjunctive}
  \end{align}
  where $\alpha$ and $\beta$ are formulas with no occurrence of the
  variable $x$. A \fterm $\phi$ is in \emph{normal form} (w.r.t. $x$)
  if it is a conjunction of \fterms $\phi_{i}$, $i \in I$, so that
  each $\phi_{i}$ either does not contain the variable $x$, or it is
  disjunctive in $x$.
\end{definition}
Due to equation~\eqref{eq:distrimpl} and since the usual distributive
laws hold in \Ha{s}, every \spositive \fterm is equivalent to a \fterm
in normal form, as witnessed by the following Lemma.
\begin{lemma}
  \label{lemma:CNF}
  Every \fterm  
  that is \fullypositive in $x$ and contains the variable
  $x$ is equivalent to a conjunction of disjunctive formulas and of a
  formula that does not contain $x$.
\end{lemma}
\begin{proof}
  By induction, we associate to each such formula a set $\tr(\phi)$ of
  disjunctive formulas and formula $c(\phi)$ so that
  \begin{align}
    \label{eq:CNF}
    \phi &\eqIpc c(\phi) \land \bigwedge \set{\delta \mid \delta \in \tr(\phi)}\,.
  \end{align}
  We let
  \ppfootnotesize
  \begin{align*}
    \tr(x) & \eqdef \set{x} \,, & c(x) & \eqdef \top\, ; \\
    \tr(\Nec[\alpha] \phi) & \eqdef \set{\Nec[\alpha]\delta \mid \delta \in
      \tr(\phi)}\,, &c(\Nec[\alpha] \phi) &\eqdef \Nec[\alpha]c(\phi)\,; \\
    \tr(\beta \vee \phi) & \eqdef \set{\beta \vee \delta \mid \delta \in
      \tr(\phi)}\,, & c(\beta \vee \phi)& \eqdef \beta \vee c(\phi)\,; \\
    \tr(\phi_{1} \vee \phi_{2}) & \eqdef \set{c(\phi_{1}) \vee \delta_{2}
      \mid \delta_{2} \in
      \tr(\phi_{2})}   & c(\phi_{1} \vee \phi_{2})& \eqdef c(\phi_{1}) \vee c(\phi_{2})\,;\\
    & \qquad \cup \set{c(\phi_{2}) \vee \delta_{1} \mid \delta_{1} \in
      \tr(\phi_{1})} \\
    & 
      \qquad \cup \,\set{\delta_{1} \vee \delta_{2} \mid
        \delta_{1} \in \tr(\phi_{1}), \delta_{2} \in
        \tr(\phi_{2})}\,,
      \\
    \tr(\phi_{1} \land \phi_{2}) & \eqdef \tr(\phi_{1}) \cup
    \tr(\phi_{2})\,, &
    c(\phi_{1} \land \phi_{2})& \eqdef c(\phi_{1}) \land c(\phi_{2})\,; \\
    \tr(\gamma \land \phi) & \eqdef \tr(\phi) \,, & c(\gamma \land \phi) &
    \eqdef \gamma \land c(\phi) \,.
  \end{align*}
  \ppnormal
  Verification that \eqref{eq:CNF} holds is routine.
\end{proof}

\subsection{\Spositive elimination}

We tackle here the problem of computing the \lfp $\mu_{x}.\phi$ of a
\fterm $\phi$ which is \emph{disjunctive} in $x$.
Recall that the formulas $\alpha$ and $\beta$ appearing in a parse
tree as leaves---according to the
grammar~\eqref{grammar:disjunctive}---do not contain the variable $x$.
We call such a formula $\alpha$ a \emph{head subformula} of $\phi$,
and such a $\beta$ a \emph{side subformula} of $\phi$, and thus we
put:
\ppfootnotesize
\begin{align*}
  \Head(\phi) & \eqdef \set{ \alpha \mid \alpha \text{ is a head
      subformula of } \phi } \,,
  &
  \Side(\phi) & \eqdef \set{ \beta \mid \beta \text{ is a side
      subformula of } \phi } \,.
\end{align*}
\ppnormal

Recall that a monotone function $f : P \rto P$ is \emph{inflating} if
$x \leq f(x)$, for all $x \in P$. 
\begin{lemma}
  The interpretation of a \stronglypositive disjunctive formula $\phi$
  as a function of $x$ is inflating.
\end{lemma}
The key observation needed to prove Proposition~\ref{prop:mudformula}
below is the following Lemma on monotone inflating functions. In the
statement of the lemma we assume that $P$ is a join-semilattice, and
that $f \vee g$ is the pointwise join of the two functions $f$ and
$g$.
\begin{lemma}
  \label{lemma:circvee}
  If $f,g:P \rto P$ are monotone inflating functions, then $f \vee g$
  and $f \circ g$ are monotone inflating and
  $\Pref_{f \vee g} = \Pref_{f \circ g}$.
  Consequently,   for any monotone function $h : P \rto P$,
  we have
  \begin{align}
    \label{eq:circvee}
    \mu.(\,f \vee g \vee h\,) & :=: \mu.(\,(f \circ g) \vee h\,)\,.
  \end{align}
\end{lemma}
\begin{proof}
It is easy to see that $f \vee g$
  and $f \circ g$ are monotone inflating, so we only verify that  $\Pref_{f \vee g} = \Pref_{f \circ g}$.
  Observe firstly that $\Pref_{f \vee g}= \Pref_{f} \cap
  \Pref_{g}$.
  If $p \in \Pref_{f \circ g}$, then $f(p) \leq f(g(p)) \leq p$ and $g(p) \leq f(g(p)) \leq p$,
  showing that $p \in \Pref_{f \vee g}$.
  Conversely, if $p \in \Pref_{f \vee g}$, then $p$ is a fixed point
  of both $f$ and $g$, since these functions are inflating. It follows
  that $f(g(p)) = f(p) = p$,
  showing $p \in \Pref_{f \circ g}$. 
  
  We have argued that $\Pref_{f \vee g}$ coincides with
  $\Pref_{f \circ g}$; this implies that
  $\Pref_{(f \circ g) \vee h} = \Pref_{f \vee g \vee h}$ and, from
  this equality, equation~\eqref{eq:circvee} immediately follows.
\end{proof}

\begin{proposition}
  \label{prop:mudformula}
  If $\phi$ is a disjunctive \fterm, then
  \begin{align}
      \label{eq:mudformula}
      \mu_{x}.\phi & := \Nec[\bigwedge_{\alpha \in \Head(\phi)}
      \alpha](\bigvee_{\beta \in \Side(\phi)} \beta)\,.
  \end{align}
\end{proposition}
\begin{proof}
  For $\psi,\chi$ \fterms, let us write $\psi \sim \chi$ when $\mu_{x}.\psi =
  \mu_{x}.\chi$.
  We say that a disjunctive formula $\psi$ is reduced (w.r.t. $\phi$)
  if either it is $x$, or it is of the form $\beta \vee x$ (or $x \vee
  \beta$) for some $\beta \in \Side(\phi)$, or of the form
  $\Nec[\alpha]x$ for some $\alpha \in \Head(\phi)$.
  A set $\Phi$ of disjunctive formulas  is \emph{reduced} if every formula in
  $\Phi$ is reduced.
  
  We shall compute a reduced set of disjunctive formulas $\Phi_{k}$
  such that $\phi \sim \bigvee \Phi_{k}$.
  Thus let $\Phi_{0} = \set{\phi}$.  If $\Phi_{i}$ is not reduced,
  then there is $\phi_{0} \in \Phi_{i}$ which is not reduced, thus of
  the form (a) $\beta \vee \psi$ (or $\psi \vee \beta$) with $\psi
  \neq x$, or (b) $\Nec[\alpha]\psi$ with $\psi \neq x$, or (c)
  $\psi_{1} \vee \psi_{2}$.  According to the case ($\ell$), with $\ell \in
  \set{a,b,c}$, we let $\Phi_{i + 1}$ be $(\Phi_{i} \setminus
  \set{\phi_{0}}) \cup \Psi_{\ell}$ where $\Psi_{\ell}$ is as follows:
  \begin{align*}
    \Psi_{a} &= \set{\beta \vee x, \psi}, 
    & \Psi_{b} & = \set{\Nec[\alpha]x, \psi}, 
    & \Psi_{c} & = \set{\psi_{1},\psi_{2}}\,.
  \end{align*}
  By Lemma~\ref{lemma:circvee}, we have $\bigvee \Phi_{i} \sim \bigvee
  \Phi_{i+1}$. Morever, for some $k \geq 0$, $\Phi_{k}$ is reduced and
  $\Phi_{k} \subseteq \set{\Nec[\alpha]x \mid \alpha \in \Head(\phi)}
  \cup \set{\beta \vee x \mid \beta \in \Side(\phi)} \cup
  \set{x}$. Consequently
  \begin{align}
    \label{eq:allsideQM}
    \mu_{x}.\phi(x) & = \mu_{x}.\bigvee \Phi_{k} \leq \mu_{x}.(x \vee
    \bigvee_{\alpha \in \Head(\phi)} \Nec[\alpha]x \vee \bigvee_{\beta
      \in \Side(\phi)} \beta \vee x )\,.
  \end{align}
  On the other hand, if $\alpha \in \Head(\phi)$, then $\phi(x) =
  \psi_{1}(x,\Nec[\alpha]\psi_{2}(x))$ for some disjunctive formulas
  $\psi_{1}$ and $\psi_{2}$, so
  \begin{align*}
    \Nec[\alpha]x & \leq \Nec[\alpha]\psi_{2}(x) \leq
    \psi_{1}(x,\Nec[\alpha]\psi_{2}(x)) = \phi(x)
  \end{align*}
  and, similarly, $\beta \vee x  \leq \phi(x)$, whenever $\beta \in \Side(\phi)$.
  It follows that
  \begin{align*}
    x \vee \bigvee_{\alpha \in \Head(\phi)} \Nec[\alpha]x \vee
    \bigvee_{\beta \in \Side(\phi)} \beta \vee x & \leq \phi(x)\,,
  \end{align*}
  whence, by taking the \lfp in both sides of the above inequality, we
  derive equality in \eqref{eq:allsideQM}.
  Finally, in order to obtain \eqref{eq:mudformula}, we compute as follows:
  \begin{align*}
    \smalllhs[7mm]{\mu_{x}.(x \vee \bigvee_{\alpha \in \Head(\phi)}
      \Nec[\alpha]x \vee
      \bigvee_{\beta \in \Side(\phi)} \beta \vee x)} \\
    & = 
    \mu_{x}.(\Nec[\alpha_{1}] \ldots \Nec[\alpha_{n}]x \vee
    (x \vee \bigvee_{\beta \in \Side(\phi)} \beta \vee x)) 
    \tag*{
by Lemma~\ref{lemma:circvee}, with  $\Head(\phi) = \set{\alpha_{1},\ldots ,\alpha_{n}}$,}
    \\
    & = \mu_{x}.(\Nec[\bigwedge_{\alpha \in \Head(\phi)} \alpha]x \vee
    (x \vee \bigvee_{\beta \in \Side(\phi)} \beta \vee x)), 
    \tag*{
      since $\Nec[\alpha_{1}] \ldots \Nec[\alpha_{n}]x =
        \Nec[\bigwedge_{i = 1,\ldots ,n}\alpha_{i}]x$,}
    \\
    & = \mu_{x}.(\Nec[\bigwedge_{\alpha \in \Head(\phi)} \alpha]( x
    \vee \bigvee_{\beta \in \Side(\phi)} \beta \vee x)) , \tag*{by
      Lemma~\ref{lemma:circvee},}
    \\
    & = \Nec[\bigwedge_{\alpha \in \Head(\phi)} \alpha]\mu_{x}.( x
    \vee \bigvee_{\beta \in \Side(\phi)} \beta \vee x) , \tag*{by
      Proposition~\ref{prop:fixrimpl},}
    \\
    & = \Nec[\bigwedge_{\alpha \in \Head(\phi)} \alpha](
    \bigvee_{\beta \in \Side(\phi)} \beta)\,.  \tag*{\qedhere}
  \end{align*}
\end{proof}

\begin{example}
  Formula~\eqref{eq:mudformula} yields
  \begin{align*}
    \mu_{x}.(\,\Nec[\alpha_{1}](\beta_{1} \vee x) \vee
    \Nec[\alpha_{2}](\beta_{2} \vee x)\,)
    & = 
    \Nec[\alpha_{1}\land \alpha_{2}](\beta_{1} \vee \beta_{2})\,.
  \end{align*}  
\end{example}

\begin{remark}
  Let $\phi$ be a disjunctive formula and consider an occurrence in
  $\phi$ of a variable $y$ distinct from $x$. Necessarily, such an
  occurrence is located in some head subformula or in some side
  subformula of $\phi$. Therefore we can map such an occurrence to an
  occurrence of the same variable within the formula on the right of
  the equality~\eqref{eq:mudformula}; notice that a \weaklynegative
  occurrence is mapped to a \weaklynegative occurrence. Since every
  occurrence of a variable $y$ in the formula on the right of
  \eqref{eq:mudformula} has a preimage through the mapping, we
  conclude the following observation, which is necessary for the
  global elimination procedure to work: \emph{if a variable $y$ is
    \weaklynegative in the disjunctive formula $\phi$, then it is
    still \weaklynegative in the formula $\mu_{x}.\phi$ as defined by
    equation~\eqref{eq:mudformula}. } Similarly, 
  if $\phi$ is \stronglypositive in $x$ and \weaklynegative in $y$,
  then $y$ is \weaklynegative in each conjunct appearing on the right
  of equation~\eqref{eq:CNF}.
\end{remark}

\subsection{\Wnegative elimination}
\label{sec:weaklynegative}

Recall that we are considering formulas $\phi$ in which every
occurrence of the variable
$x$ is  positive.
Therefore, if $\phi$ is \wnegative in
$x$, 
then we can write
\begin{align}
  \label{eq:wndecomposition}
  \phi(x) & = \psi_{0}(\psi_{1}(x), \ldots ,\psi_{n}(x))\,,
\end{align}
for \fterms $\psi_{0}(y_{1},\ldots ,y_{n})$ and $\psi_{i}(x)$, $i =
1,\ldots ,n$, such that: (a) all the variables $y_{i}$ are negative in $\psi_{0}$; (b) for $i = 1,\ldots ,n$,
$x$ is negative  $\psi_{i}$.

\begin{proposition}
  \label{prop:weaklynegative}
  Let $\phi$ be a formula which is \weaklynegative in $x$.
  Let $\langle \nu_{1},\ldots ,\nu_{n}\rangle$ be a collection of
  \fterms denoting the greatest solution of the system of equations
  $\set{ y_{i} = \psi_{i}(\psi_{0}(y_{1},\ldots ,y_{n})) \mid i =
    1,\ldots ,n }$. Then $\psi_{0}(\nu_{1},\ldots ,\nu_{n})$ is a
  formula equivalent to $\mu_{x}.\phi(x)$.
\end{proposition}
\begin{proof}
  Let $v : \Vars \setminus \set{x,y_{1},\ldots ,y_{n}} \rto H$ be a
  partial valuation into a \Ha $H$, put $\Pf_{0} = \eval{\psi_{0}}$
  and, for $i = 1,\ldots ,n$, $\Pf_{i} = \eval{\psi_{i}}$. Then
  $\Pf_{0}$ is a monotone function from $ [ H^{op}]^{n}$ to $H$. Here
  $H^{op}$ is the poset with the same elements as $H$ but with the
  opposite ordering relation. Similarly, for $1 \leq i \leq n$,
  $\Pf_{i} : H \rto H^{op}$.  If we let $\bar{\Pf} = \langle \Pf_{i} \mid
  i = 1,\ldots ,n \rangle \circ \Pf_{0}$, then $\bar{\Pf}
  :{[H^{op}]}^{n} \rto {[H^{op}]}^{n}$.
  We exploit next the fact that $(\cdot)^{op}$ is a functor, so that
  $f^{op} : P^{op} \rto Q^{op}$ is the same monotone function as $f$,
  but considered as having distinct domain and codomain.  Then, using
  \roll, we can write
  \ppsmall
  \begin{align}
    \label{eq:munu}
    \mu. (\,\Pf_{0} \circ \langle \Pf_{i} \mid i = 1,\ldots ,n \rangle
    \,) & = \Pf_{0}(\,\mu. (\,\langle \Pf_{i} \mid i = 1,\ldots ,n
    \rangle 
    \circ \Pf_{0}\,)\,)
    = \Pf_{0}(\,\mu.\bar{\Pf}\,) = \Pf_{0}(\,\nu.\bar{f}^{op}\,)\,,
  \end{align}
  \ppnormal
  since the \lfp of $f$ in $P^{op}$ is the \gfp of $f^{op}$ in
  $P$. That is, if we consider the function $\langle \Pf_{i} \mid i =
  1,\ldots ,n \rangle \circ \Pf_{0}$ as sending a tuple of elements of
  $H$ (as opposite to $H^{op}$) to another such a tuple, then
  equation~\eqref{eq:munu} proves that a formula denoting the \lfp of
  $\phi$ is constructible out of formulas for the greatest solution of
  the system mentioned in the statement of the proposition.  
\end{proof}

As far as computing the greatest solution of the system mentioned in
the proposition, this can be achieved by using the \Bekic elimination
principle 
(see Lemma~\ref{lemma:bekic}).  This principle implies that
solutions of systems can be constructed from solutions of linear
systems, i.e. from usual parametrized \fp{s}. In our case, as
witnessed by equation~\eqref{eq:nuexists}, these parametrized \gfp{s}
are computed by substituting $\top$ for the \fp variable.

\begin{example}
  \label{ex:binaryjoin}
  Consider  the \wnegative $\phi$ defined by
  \begin{align}
    \label{def:jointwo}
    \phi(x) & \eqdef ((x \impl c) \impl a)\, \vee \,((x \impl d) \impl b) \,.
  \end{align}
  We can take then
  \begin{align*}
    \psi_{0}(y_{1},y_{2}) & \eqdef y_{1}\impl a \vee y_{2} \impl b\,,
    \;\;\psi_{1}(x) \eqdef x \impl c\,, \;\;\psi_{2}(x) = x \impl d
    \,.
  \end{align*}
  The system of equations whose greatest solution we need to compute
  is
  \ppfootnotesize
  $$
  \begin{array}{cc}
    \begin{array}{l@{\,}l}
      y_{1} &
      = \psi_{1}(\psi_{0}(y_{1},y_{2})) 
      =(y_{1}\impl a \vee
      y_{2} \impl
      b) \impl c\,, 
    \end{array}
    \begin{array}{l@{\,}l}
      y_{2} & = \psi_{2}(\psi_{0}(y_{1},y_{2})) 
      = (y_{1}\impl a \vee y_{2} \impl b) \impl d \,.
    \end{array}
  \end{array}
  $$
  \ppnormal
  The \Bekic elimination principle is used to find this solution:
  \pptiny
  \begin{align*}
    \nu_{y_{2}}.\psi_{2}(\psi_{0}(y_{1},y_{2})) & =
    \nu_{y_{2}}.
    ((y_{1}\impl a) \vee
    (y_{2} \impl b)) \impl d 
    = ((y_{1}\impl a) \vee
    (\top \impl b)) \impl d 
    = ((y_{1}\impl a) \vee
    b) \impl d \,, \\
    \nu_{1} & =
    \nu_{y_{1}}.\psi_{1}(\psi_{0}(y_{1},\nu_{y_{2}}.\psi_{2}(\psi_{0}(y_{1},y_{2}))))
    =
    \psi_{1}(\psi_{0}(\top,\nu_{y_{2}}.\psi_{2}(\psi_{0}(\top ,y_{2})))) \\
    & = (\top \impl a \vee (((\top \impl a) \vee b)
    \impl d) \impl
      b) \impl c 
      = (a \vee (((a \vee b) \impl d) \impl
      b)) \impl c\,, \\
      \nu_{2} & = ((\nu_{1}\impl a) \vee b) \impl d\,.
    \end{align*}
    \ppnormal
  Then, by \eqref{eq:rolling}, we have
  $\mu_{x}.\phi(x) = \nu_{1}\impl a \vee \nu_{2} \impl b$.
\end{example}
In the next Section, Proposition~\ref{prop:convergesforwefs} shall
provide an alternative of the \lfp of a \weaklynegative formula $\phi$
by means of approximants.

\section{Upper bounds for closure ordinals}
\label{sec:generalUpperBounds}

The closure ordinal of $\phi(x) \in \FIpc$ is the least integer $n$
for which we can write $\mu_{x}.\phi(x) = \phi^{n}(\bot)$. In view of
the proof of Proposition~\ref{prop:total}, the closure ordinal always
exists, for each intuitionistic formula $\phi(x)$ positive on
$x$. Closure ordinals yield a representation of \lfp{s} of formulas
alternative to the one presented in the previous Section.  
Such representation can be exploited notationally, as in $\mu$-calculi
with explicit approximations \cite{damGurov}.
Also it can be exploited computationally because of the reduced space
requirements, at least when variable sharing is used. Finally, it can be
exploited to provide axiomatizations.
In this Section we begin the study of (finite)
closure ordinals.

\subsection{General results}
In this Section all the posets we consider shall have a least element,
denoted by $\bot$ as usual.
We say that a monotone function $f: P \rto P$ \emph{converges in $n$
  steps to its \lfp} if $f^{n+1}(\bot) = f^{n}(\bot)$ or,
equivalently, if $\Mu{f} = f^{n}(\bot)$; 
in such a case the least of those integers
$n$ is called the \emph{closure ordinal of
$f$} and it is denoted by $\cl(f)$.  
We informally call the $f^{n}(\bot)$, $n \geq
0$, the \emph{approximants} (or \emph{approximations}) of (the \lfp
of) $f$.
If
$f : Q \times P^{n} \rto P^{k}$ is a monotone funtion
and $\set{i_{1}< i_{2} < \ldots < i_{k}} \subseteq \set{1,\ldots ,n}$,
then we write $\cl[(x_{i_{1}},\ldots ,x_{i_{k}})](f) \leq n$ if, for
each $q \in Q$ and $\vec{p} \in P^{n-k}$,
$\cl(f_{(q,\vec{p})}) \leq n$, where
$f_{(q,\vec{p})} : P^{k} \rto P^{k}$ is the monotone function obtained
from $f$ by fixing $q \in Q$ and evaluating all the variables $x_{j}$
with $j \not\in \set{i_{1},\ldots ,i_{k}}$ by means of the vector
$\vec{p}$.

\bigskip

The next propositions suggest how to compute convergence of monotone
functions based on the properties of \lfp{s} that we have introduced
in Section~\ref{sec:notation}.

\begin{proposition}[Convergence for \eqref{eq:rolling}]
  \label{prop:convrolling}
  Let $f : P \rto Q$ and $g: Q \rto P$ be monotone functions. If
  $\Mu{(f \circ g)} = (f \circ g)^{n}(\bot)$, then
  $\Mu{(g \circ f)} = (g \circ f)^{n+1}(\bot)$.  Therefore
  $\cl(f \circ g) \leq 1 + \cl(g \circ f)$.
\end{proposition}
\begin{proof}
  We observe that
  \begin{align*}
    \Mu{(g \circ f)}
    & = g (\Mu{(f \circ g)}) = g \circ (f \circ g)^{n}(\bot) 
    \leq g \circ (f \circ g)^{n}(f(\bot)) = (g \circ f)^{n+1}(\bot) \,.
  \end{align*}
  Since the converse inclusion always holds, we have proved the
  proposition.
\end{proof}

\begin{example}
  \label{ex:convrollingtight}
  Consider $\phi(x) \eqdef (x \impl b) \impl a$. By using
  Proposition~\ref{prop:weaklynegative}
  (with $\psi_{0}(y_{1}) \eqdef (y_{1}\impl a)$ and
  $\psi_{1}(x) \eqdef (x\impl b)$)
  we know that
  \begin{align*}
    \mu_{x}.\phi(x) & = (\nu_{x}.(x\impl a) \impl b) \impl  a = ((\top
    \impl a) \impl b) \impl  a = (a \impl b) \impl a\,.
  \end{align*}
  Otherwise, we can combine Propositions~\ref{prop:phitop}
  and~\ref{prop:convrolling} to deduce
  $\mu_{x}.\phi(x) = \phi^{2}(\bot)$. Indeed, a direct computation of
  the approximants 
  yields
  \begin{align*}
    \phi(\bot) & = a\,, \;\phi^{2}(\bot) = (a \impl b) \impl a\,.    
  \end{align*}
  This example shows that the bound on the convergence given in
  Proposition~\ref{prop:convrolling} is tight, since 
  the equality $\phi^{2}(\bot) = \phi(\bot)$ only holds for
  arbitrary $a$ and $b$ whenever
  $H$ is a Boolean algebra. As a matter of fact, note that this
  equality is Peirce's law
  \begin{align*}
    (a \impl b) \impl a
    & = a\,,
  \end{align*}
  which forces a \Ha to be Boolean.
\end{example}

\begin{proposition}[Convergence for \eqref{eq:bekic}]
  \label{prop:bekicconvergence}
  Let $\langle f,g\rangle : P\times Q \rto P\times Q$ be a monotone
  mapping.  Put $h(x) \eqdef f(x,\mu_{y}.g(x,y))$. Let $m,n \geq 0$ be
  such that $\mu_{y}.g(x,y) = g_{x}^{m}(\bot)$ for each $x \in P$ and
  $\mu_{x}.h(x) = h^{n}(\bot)$.  Then
  \begin{align}
    \label{eq:bekicconvergence}
    \Mu{\langle f,g\rangle} & = \langle f,g\rangle^{(n+1)(m+1)
      -1}(\bot,\bot)\,.\
  \end{align}
  That is,
  $\cl(\langle f,g\rangle) \leq (\cl[y](g) + 1)(\cl(h) +1) -1$.
\end{proposition}
\begin{proof}
  Let us define by induction the following sequences:
  \begin{align*}
    \f_{0} & = \g_{0} = \bot, &
    \f_{i+1} & = f(\f_{i},\g_{i})\,, & \g_{i+1} & = g(\f_{i},\g_{i})\,,\\
    \k_{0} & = \h_{0} = \bot\,, &
    \k_{i+1} &= (g_{\h_{i}})^{m}(\bot)\,, & \h_{i + 1} & = f(\h_{i},\k_{i + 1})\,.
  \end{align*}
  Notice first that, for each $i \geq 0$,
  $\langle f,g\rangle^{i}(\bot,\bot) = \langle \f_{i},\g_{i}\rangle$, On
  the other hand, we have
  \begin{align*}
    \h_{i + 1} & = f(\h_{i},\k_{i+1}) = f(\h_{i},(g_{\h_{i}})^{m}(\bot)) =
    f(\h_{i},\mu_{y}.g(\h_{i},y)) = h(\h_{i})\,,
  \end{align*}
  so, by a straightforward induction, we obtain that $\h_{i} =
  h^{i}(\bot)$. Then, by the \Bekic property, 
  \begin{align*}
    \Mu{\langle f,g\rangle} & = \langle
    h^{n}(\bot),(g_{h^{n}(\bot)})^{m}(\bot)\rangle = \langle
    \h_{n},\k_{n+1} \rangle.
  \end{align*}
  \begin{claim*}
    Let $\psi : \N \rto \N$ be any function.  For each $i \geq 0$,
    \begin{enumerate}
    \item $\h_{i} \leq \f_{\psi(i)}$ implies 
      $\k_{i+1} \leq \g_{\psi(i) + m}$;
    \item $\h_{i} \leq \f_{\psi(i)}$ implies
      $\h_{i} \leq \f_{\psi(i) + m +1}$.
    \end{enumerate}
  \end{claim*}
  \begin{proofofclaim}
    (1) Let us suppose that $\h_{i} \leq \f_{\psi(i)}$ and prove that
    $(g_{\h_{i}})^{\ell}(\bot) \leq \g_{\psi(i) + \ell}$ for
    $\ell = 0,\ldots ,m$.
    This relation trivially holds for $\ell = 0$ and, 
    supposing it holds for $\ell$,
    \begin{align*}
      (g_{\h_{i}})^{\ell + 1}(\bot) & =
      g_{\h_{i}}(g_{\h_{i}}^{\ell}(\bot))
      \leq g_{\h_{i}}(\g_{\psi(i) + \ell}) = g(\h_{i}, \g_{\psi(i) + \ell}) \\
      & \leq g(\f_{\psi(i)}, \g_{\psi(i) + \ell}) \leq g(\f_{\psi(i) +
        \ell}, \g_{\psi(i) + \ell}) = \g_{\psi(i) + \ell +1}\,.
    \end{align*}
    Thus, for $\ell = m$, we have $\k_{i+1} = (g_{\h_{i}})^{m}(\bot) \leq
    \g_{\psi(i) + m}$.

    (2) If we suppose $\h_{i} \leq \f_{\psi(i)}$, then
    $\k_{i+1} \leq \g_{\psi(i) + m}$ by (1), and
    \begin{align*}
      \h_{i + 1}  = f(\h_{i},\k_{i +1}) & \leq f(\f_{\psi(i)},\g_{\psi(i) +
        m})
      \leq f(\f_{\psi(i) + m},\g_{\psi(i) + m}) = \f_{\psi(i) + m
        +1}\,.
    \end{align*}
  \end{proofofclaim}
  If now we let  $\psi(i) \eqdef i(m +1)$, then 
  $\h_{i} \leq \f_{\psi(i)}$, for all $i \geq 0$, by induction on $i$
  and using part (2) of the Claim.
  Then 
  we 
  deduce that
  \begin{align*}
    \Mu{\langle f,g\rangle} & = \langle \h_{n},\k_{n+1}\rangle \\
    & \leq
    \langle \f_{\psi(n)},\g_{\psi(n) + m}\rangle
    \leq  \langle
    \f_{\psi(n) + m },\g_{\psi(n) + m}\rangle = \langle
    f,g\rangle^{\psi(n) + m}(\bot,\bot)\,,  
  \end{align*}
  showing that the function $\langle f,g \rangle$ converges to its
  \lfp in $\psi(n) + m = (n+1)(m+1) -1$
  steps.
\end{proof}
\begin{example}
  We argue that the upper bound $(n+1)(m+1)-1=(m+1)n+m$ given in
  Proposition~\ref{prop:bekicconvergence} is tight.
  For $m,n \geq 1$, let $P$ and $Q$ be respectively the $n+1$-element
  chain $\set{0 < 1<\ldots<n}$ and the $(n+1)m +1$-element chain
  $\set{0<1<\ldots<(n+1)m}$.  On these chains define the successor
  function $s$ by $s(x) = x +1$ if $x \neq \top$ and, otherwise,
  $s(\top) = \top$.
  If $y \in Q$, then it can be written in the form $zm + k$ for some 
  $0\leq k<m$ and $0\leq z\leq n+1$.
  Define the mappings $f\colon P \times Q \to P$
  and $g\colon P\times Q \to Q$ as follows:
  \begin{align*}
    f(x,zm+k) & = 
    \begin{cases}
      x\,, & \text{if $z\leq x$,}\\
      s(x)\,, & \text{otherwise},
    \end{cases} \\
    g(x,zm+k) & =
    \begin{cases}
      xm + k + 1\,, & \text{if $z \leq x$,} \\
      (x +1)m\,, & \text{otherwise}.
    \end{cases} 
  \end{align*}
  where $0\leq k<m$.  Both $f$ and $g$ are monotone (for example,
  $f(x,zm+k) = \max(x,\min(z,s(x)))$).  Consider now the mapping
  $\langle f,g\rangle : P\times Q \rto P\times Q$ and recall that
  $h(x) = f(x,\mu_{y}.g(x,y))$. The following holds:
  \ppfootnotesize
  \begin{align*}
    \mu_{y}.g(x,y) & = (x+1)m = [g_{x}]^{m}(\bot) \,,
    \quad h(x) = f(x,(x+1)m) = s(x)\,, 
    \quad \mu_{x}.h(x) = n = h^{n}(\bot)\,.
  \end{align*}
  \ppnormal
  It follows that
  $\Mu{\langle f,g\rangle}=(n, \mu_{y}.g(n,y))=(n, (n+1)m)$. Finally
  observe that
  \ppfootnotesize
  \begin{align*}
    \langle
    f,g\rangle^{(m+1)(n+1)-2}(\bot,\bot) & =(n,nm+m-1)
    < (n,
    (n+1)m)=\langle f,g\rangle^{(m+1)n+m}(\bot,\bot).
  \end{align*}
   \ppnormal
\end{example}

\begin{proposition}[Convergence for \eqref{eq:diag}]
  Let $f : P \times P \rto P$ be a monotone function and put
  $h(x) \eqdef \mu_{y}.f(x,y)$.  Let $n,m \geq 0$ be such that
  $h(x) = f_{x}^{m} (\bot)$, for each $x \in P$, and
  $\mu_{x}.h(x) = h^{n}(\bot)$.  Then
  $\mu_{x}.f(x,x) = f^{nm}(\bot,\bot)$.  That is,
  $\cl(f \circ \Delta) \leq \cl(h)\cl[y](f)$.
\end{proposition}
\begin{proof}
  An easy inspection shows that  $\cl(f\circ \Delta) = \cl(\langle
  f,f\rangle)$ and
  hence we refer back to
  Proposition~\ref{prop:bekicconvergence}. Consider
  $\f_{i},\g_{i},\k_{i},\h_{i}$ as defined in the proof of that
  Proposition. Here we have $g = f$, so $\g_{i} = \f_{i}$ for each
  $i \geq 0$, and moreover
  \begin{align*}
    \h_{i +1} & = f(\h_{i},\mu_{y}.g(\h_{i},y)) 
    = f(\h_{i},\mu_{y}.f(\h_{i},y)) = \mu_{y}.f(\h_{i},y) =
    \mu_{y}.g(\h_{i},y) = \k_{i+1}\,,
  \end{align*}
  so $\h_{i} = \k_{i}$ for each $i \geq 0$. According to the Claim in
  the proof of Proposition~\ref{prop:bekicconvergence},
  $\h_{i} \leq f_{\psi(i)}$ implies
  $\kappa_{i +1} \leq \g_{\psi(i) + m}$; 
  that is, $\h_{i +1} \leq \f_{\psi(i) + m}$ since $f=g$.
  Therefore, letting $\psi(i) \eqdef im$, we deduce
  $\h_{i} \leq \f_{\psi(i)}$ for all $i \geq 0$ which 
  implies that
  \begin{align*}
    \mu_{x}.f(x,x) &= \mu_{x}.\mu_{y}.f(x,y)\,, \tag*{by \eqref{eq:diag},}\\
    & = \mu_{x}.h(x)\,,
    \tag*{since $h(x) = \mu_{y}.f(x,y)$,}\\
    & = \h_{n}\,, \tag*{since $\h_{n} = h^{n}(\bot)$ and we assume
      that $\mu_{x}.h(x) =  h^{n}(\bot)$,}\\
    & \leq \f_{nm} \,, \tag*{since $\h_{n} \leq
      \f_{\psi(n)}$ and $\psi(n) = nm$,}
  \end{align*}
  as needed.
\end{proof}

\subsection{Results for \Ha{s}}
In many cases, formula \eqref{eq:bekicconvergence} 
given in Proposition~\ref{prop:bekicconvergence} does not yield a tight upper
bound.  
In particular this happens when we want to estimate the
convergence of \weaklynegative formulas whose \lfp{s} can be computed by
using the \Bekic property, as we have seen
in the previous Section~\ref{sec:weaklynegative}. 

In order to improve the upper bound given in
\eqref{eq:bekicconvergence}, we need the following observation.
\begin{lemma} 
  \label{lemma:pushdown}
  Let $\langle f,g\rangle : P\times Q \rto P\times Q$ be a monotone
  mapping,  put $h(x) \eqdef f(x,\mu_{y}.g(x,y))$, let $m,n \geq 0$ be
  such that $\mu_{y}.g(x,y) = g_{x}^{m}(\bot)$ for each $x \in P$ and
  $\mu_{x}.h(x) = h^{n}(\bot)$.
  Under these hypothesis we
  have 
  \begin{align*}
    \pi_{1}(\mu.\langle f,g \rangle) & = \pi_{1} (\langle f,g\rangle^{n(m+1)}(\bot,\bot))\,.  
  \end{align*}
\end{lemma}
\begin{proof}
  Using the same notation as in the proof of
  Proposition~\ref{prop:bekicconvergence}, it is enough to observe
  that
  \begin{align*}
    \mu.\langle f,g\rangle & \leq \langle \f_{\psi(n)},\g_{\psi(n) +
      m}\rangle \leq \mu.\langle f,g\rangle\,,
  \end{align*}
  with $\psi(n) = n(m+1)$, so
  $\pi_{1} (\langle f,g\rangle^{\psi(n)}) = \f_{\psi(n)} =
  \pi_{1}(\mu.\langle f,g\rangle)$.
\end{proof}

By using the lemma, we are going to obtain the tight upper bound for the \lsol of system
of equations used for \weaklynegative \fterms.

\begin{proposition}
  \label{prop:convergencesystem}
  Consider a monotone
  $\langle f_{1},\ldots ,f_{k}\rangle : Q \times P^{k} \rto P^{k}$ and
  suppose that, for each $g : Q \times P^{k} \rto P$ in the cone
  generated by the functions
  $\set{ f_{1},\ldots ,f_{k} } \cup \set{\bot}$,
  $\cl[x_{i}](g) \leq 1$ for each $i =1,\ldots ,k$.  Then
  $\mu.\langle f_{1},\ldots ,f_{k} \rangle_{q} \leq \langle
  f_{1},\ldots ,f_{k} \rangle_{q}^{k}(\bot)$ for each $q \in Q$ or,
  said otherwise,
  $\cl[(x_{1},\ldots ,x_{k})](\langle f_{1},\ldots ,f_{k} \rangle)
  \leq k$.
\end{proposition}
\begin{proof}
  The proof is by induction on $k \geq 1$. When $k = 1$ then,
  $\cl[x_{1}](f_{1}) \leq 1$ by assumption.

  Now suppose that $k > 1$ and that the property holds for all
  motone functions
  $\langle f_{i_{1}},\ldots ,f_{i_{l}}\rangle : Q \times P^{\ell} \rto
  P^{\ell}$ with $\ell < k$.

  By the induction hypothesis,
  $\cl[(x_{2},\ldots ,x_{k})](\langle f_{2},\ldots ,f_{k}\rangle)
  \leq k -1$.  
  For each $q\in Q$ consider the function $h_q$ defined
  by 
  $h_q(x_{1}) \eqdef f_{1}(q,x_{1},\langle f_{2},\ldots
  ,f_{k}\rangle_{(q,x_{1})}^{k}(\bot))$; $h_q$ belongs to the cone
  generated by $\set{ f_{1},\ldots ,f_{k} } \cup \set{\bot}$ and
  therefore $\cl[x_{1}](h_q) \leq 1$ by assumption.  We can therefore
  apply Lemma~\ref{lemma:pushdown} (with 
  $f=h_q$, $g=\langle f_{2},\ldots ,f_{k}\rangle$, $n =1$ and
  $m = k-1$) to deduce that, for each $q \in Q$,
  \begin{align*}
    \pi_{1}(\mu.\langle f_{1},\ldots ,f_{k}\rangle_{q})\leq
      \pi_{1}(\langle f_{1},\ldots ,f_{k}\rangle_{q}^{1\cdot(k-1 +1)})(\bot) = \pi_{1}(\langle
    f_{1},\ldots ,f_{k}\rangle_{q}^{k}(\bot))\,.
  \end{align*}
  In a similar way we deduce 
  \begin{align*}
    \pi_{i}(\mu.\langle f_{1},\ldots ,f_{k}\rangle_{q}) & \leq
    \pi_{i}(\langle f_{1},\ldots ,f_{k}\rangle_{q}^{k}(\bot))\,,
  \end{align*}
  for each $i = 1,\ldots ,k$, and therefore
  $\mu.\langle f_{1},\ldots ,f_{k}\rangle_{q} \leq \langle
  f_{1},\ldots ,f_{k}\rangle_{q}^{k}(\bot)$.
\end{proof}

To see that the bound given in the previous proposition is tight it is
enough to compute the least solution of the system of equations
\begin{align*}
  & \set{ x_{i} = \set{a_{i}} \cup x_{i-1\!\!\!\!\!\mod k} \mid i = 0,\ldots ,k-1 }\,,
\end{align*}
in the powerset of $P(\set{a_{1},\ldots ,a_{k}})$.
\medskip

We can finally give a better upper bound to closure ordinals of
\weaklynegative \fterms.
\begin{proposition}
  \label{prop:convergesforwefs}
  Let $\phi(x)$ be a \weaklynegative \fterm, so that we have a
  decomposition of the form \eqref{eq:wndecomposition}.  Then
  $\phi(x)$ converges to its \lfp in at most $n + 1$ steps.
\end{proposition}
\begin{proof}
  By combining Propositions~\ref{prop:phitop}
  and~\ref{prop:convergencesystem}, we have
  \begin{align}
    \nu.(\langle \psi_{i} \mid
    i = 1,\ldots ,n \rangle \circ \psi_{0}) & = (\langle \psi_{i} \mid
   i = 1,\ldots ,n \rangle \circ \psi_{0})^{n}(\top)\,.
   \label{eq:munu2}
 \end{align}
 Considering that
 \begin{align*}
   \mu.\phi & = \mu.(\psi_{0} \circ \langle \psi_{i} \mid
   i = 1,\ldots ,n \rangle) 
   = \psi_{0} (\nu.(\langle \psi_{i} \mid
      i = 1,\ldots ,n \rangle \circ \psi_{0}))
   \end{align*}
   we can use equation \eqref{eq:munu2} and
   Proposition~\ref{prop:convrolling} to deduce that
   \begin{align*}
     \mu.\phi & = (\psi_{0} \circ \langle \psi_{i} \mid
     i = 1,\ldots ,n \rangle)^{n+1}(\bot) \,.
     \tag*{\qedhere}
   \end{align*}
\end{proof}

We can expect that other formulas for \fp{s} have a counterpart with
closure ordinals. This is the case for
equation~\eqref{eq:muconjunction}.  To give an account of it, we
firstly prove a a Lemma.
\begin{lemma}
  \label{lemma:conjunctionconvergence}
  Let $H$ be a \HA and let $\Pf$ and $\g$ be monotone polynomials
  on $H$. For every pair of natural numbers $n,m$ such that
  $n + m \geq 1$,
  $\Pf^n(\bot) \land \Pg^m(\bot)\leq (\Pf\land \Pg)^{n +
    m-1}(\bot)$.
\end{lemma}
\begin{proof}
  Let $\Ph$ denote the polynomial $\Pf \wedge \Pg$ on $H$. We prove
  the result by induction on $k = n + m\geq 1$.

  If $n + m = 1$, then either $n = 0$ or $m = 0$. In this case either
  $\Pf^{n}(\bot) = \bot $ or $\Pg^{m}(\bot) = \bot$, so
  $\Pf^{n}(\bot) \land \Pg^{m}(\bot) = \bot$, so the result is
  obvious.
  
  Now suppose that the result holds for any pair of numbers $n',m'$
  such that $1 \leq n'+m' \leq k$. Let $m$ and $n$ be such that
  $m+n=k+1$. The following holds:
  \begin{align*}
    \Pf^n(\bot)\wedge \Pg^{m}(\bot) = &\Pf^n(\bot)\wedge \Pg^{m}(\bot)\wedge\Pf^n(\bot)\wedge \Pg^{m}(\bot)\\
    &\leq\Pf(\Pf^{n-1}(\bot)\wedge \Pg^m(\bot))\wedge
    \Pg(\Pf^n(\bot)\wedge \Pg^{m-1}(\bot))\,,
    \tag*{using strongness,}
    \\
    &\leq \Pf(\Ph^{k-1}(\bot)) \wedge
    \Pg(\Ph^{k-1}(\bot))=\Ph^k(\bot)=\Ph^{n+m-1}(\bot)\,.  \tag*{\qedhere}
  \end{align*}
\end{proof}
Next we show that $\cl(\Pf \land \Pg) < \cl(\Pf) + \cl(\Pg)$.
This relation holds when $\cl(\Pf) + \cl(\Pg) > 0$; in order to settle
trivial cases, we let $\Ph^{k}(\bot) = \bot$ for $k < 0$ in the
statement of the  Proposition below.
\begin{proposition}
  \label{prop:conjunctionconvergence}
  Let $H$ be an \HA. If $\Pf$ and $\Pg$ are monotone polynomials on $H$
  such that $\mu_{x}.\Pf(x) = \Pf^{m}(\bot)$ and
  $\mu_{x}.\Pg(x) = \Pg^{n}(\bot)$, then
  $\mu_{x}.(\Pf\wedge \Pg)(x) = (\Pf\wedge \Pg)^{m+n-1}(\bot)$.  That is,
  $\cl(\Pf \land \Pg) \leq \cl(\Pf) + \cl(\Pg) -1$.
\end{proposition}
\begin{proof}
Let $\Ph(x) = \Pf(x) \land \Pg(x)$ and compute as follows:
\begin{align*}
  \Ph^{n +m -1}(\bot)
  & \leq \mu_{x}.\Ph(x) = \mu_{x}.(\Pf(x) \land \Pg(x)) \\
  & = \mu_{x}.\Pf(x) \land \mu_{x}.\Pg(x)\,,
  \tag*{by Proposition~\ref{cor:distrconj},} \\
  & = \Pf^{n}(\bot) \land \Pg^{m}(\bot) \\
  & \leq \Ph^{n+m -1}(\bot) \,,
  \tag*{by Proposition~\ref{lemma:conjunctionconvergence},}
\end{align*}
so we have 
the equality $\mu_{x}.\Ph(x)=\Ph^{n +m -1}(\bot)$.
\end{proof}

\begin{proposition}
  \label{prop:conjubtight}
  The upper bound $m+n-1$ given in
  Proposition~\ref{prop:conjunctionconvergence} is tight.
\end{proposition}
\begin{proof}
  Observe that if $H$ is a \HA which is a chain, then
  $x \impl a = \top$, if $x \leq a$, and $x \impl a = a$, otherwise.
  If $H$ is such an \HA which contains the chain
  $\bot \leq a_{0}<a_1<a_2<\ldots<a_{k-1} < a_k=\top$, let 
  \begin{align*}
    \Pf_{j}(x) & \eqdef (x \impl a_{j -1}) \impl a_{j}\,, \quad
    \text{for $j =1,\ldots ,k -1$.}
  \end{align*}
  We have then, for each $i,j$ with $0\leq i\leq k$ and $1 \leq j < k$,
  \begin{align*}
    \Pf_{j}(a_{i}) & = (a_{i} \impl a_{j-1}) \impl a_{j} =
    \begin{cases}
      a_{j} & i < j\,, \\
      \top & i \geq j\,.
    \end{cases}
  \end{align*}
  Define then
  \begin{align*}
    \Pf_{a_{0},a_{1},\ldots ,a_{k-1}}(x) & \eqdef \bigwedge_{j =1,\ldots ,k-1}
    \Pf_{j}(x)\,.
  \end{align*}
  \begin{claim*} For each $i = 1,\ldots ,k$ we have 
    \begin{align*}
      \Pf_{a_{0},a_{1},\ldots a_{k-1}}^{i}(\bot) & = a_{i}\,.
    \end{align*}
  \end{claim*}
  \begin{proof}[Proof of Claim.]
    The relation trivially holds for $i = 1$. Assuming it
    holds for $i$, we have
    \begin{align*}
      \Pf^{i+1}_{a_{0},a_{1},\ldots ,a_{k-1}}(\bot) & = \bigwedge_{j
        =1,\ldots ,k-1} \Pf_{j}(a_{i})
      = \bigwedge_{i < j \leq k -1} \Pf_{j}(a_{i}) = \bigwedge_{i < j
        \leq k-1} a_{j} = a_{i +1}\,.
    \end{align*}
    Observe that the above relation holds alse when $i+1 = k$, in
    which case $\set{j \mid i < j \leq k-1} = \emptyset$, so the meet
    above is empty, so equal to $\top = a_{k}$.
  \end{proof}
  It follows from the Claim that
  $\mu_{x}.\Pf_{a_{0},a_{1},\ldots ,a_{k-1}}(x) =
  \Pf^{k}_{a_{0},a_{1},\ldots ,a_{k-1}}(\bot) = \top > a_{k-1} =
  \Pf^{k-1}(\bot)$.

  Now assume that $H$ contains the chain
  $\bot \leq a_{0}<a_1<a_2<\ldots<a_{m+n-2}<a_{m+n-1}=\top$. We have then
  \begin{align*}
    \Pf_{a_{0},a_{1},\ldots ,a_{n+m -2}}(x) & =
    \Pf_{a_{0},a_{1},\ldots ,a_{n-1}}(x) \land
    \Pf_{a_{n-1},a_{n},\ldots ,a_{n + m -2}}(x)\,, 
  \end{align*}
  with
  \begin{align*}
    \mu_{x}.\Pf_{a_{0},a_{1},\ldots ,a_{n-1}}(x) & =
    \Pf^{n}_{a_{0},a_{1},\ldots
      ,a_{n-1}}(\bot)\,,\\
    \mu_{x}.\Pf_{a_{n-1},a_{n},\ldots ,a_{n + m -2}}(x) & =
    \Pf^{m}_{a_{n-1},a_{n},\ldots ,a_{n + m
        -2}}(\bot)\,,\\
    \mu_{x}.\Pf_{a_{0},a_{1},\ldots ,a_{n+m-2}}(x) & = \Pf^{n+m
      -1}_{a_{0},a_{1},\ldots ,a_{n+m -2}}(\bot) > \Pf^{n+m
      -2}_{a_{0},a_{1},\ldots ,a_{n+m -2}}(\bot)\,.  \tag*{\qedhere}
  \end{align*}
\end{proof}
Finally, we provide a tight upper bound for closure ordinals of
disjunctive formulas.
\begin{proposition}
    \label{prop:convergencedf}
  If $\phi$ is a disjunctive formula, then
  \begin{align}
    \label{eq:convergencedf}
    \mu_{x}.\phi(x)  & = \phi^{n +1}(\bot)\,,
  \end{align}
  where $n$ is the cardinality of the set $\Head(\phi)$.
\end{proposition}
\begin{proof}
 By Proposition~\ref{prop:mudformula}  we know that $\mu_{x}.\phi(x)= \Nec[\bigwedge_{i = 1,\ldots, n} \alpha_{i}](\bigvee_{\beta \in
      \Side(\phi)} \beta) $.
      We have seen that, for $\alpha \in \Head(\phi)$, $\Nec[\alpha]x \leq
  \phi(x)$ and, similarly, $\beta \vee x \leq \phi(x)$ for $\beta \in \Side(\phi)$.
  Thus we have
  \begin{align*}
    \bigvee_{\beta \in \Side(\phi)} \beta
    & = \bigvee_{\beta \in \Side(\phi)} \beta \vee \bot \leq \phi(\bot)\,.
  \end{align*}
  Let $\Head(\phi) = \set{\alpha_{1},\ldots ,\alpha_{n}}$ and suppose
  that
  \begin{align*}
    \Nec[\alpha_{i}] \ldots \Nec[\alpha_{1}](\bigvee_{\beta \in
      \Side(\phi)} \beta) \leq \phi^{i +1}(\bot)\,.
  \end{align*}
  Then
  \begin{align*}
   \Nec[\alpha_{i+1}] \Nec[\alpha_{i}] \ldots \Nec[\alpha_{1}](\bigvee_{\beta \in
      \Side(\phi)} \beta) \leq \Nec[\alpha_{i+1}](\phi^{i +1}(\bot))
    \leq\phi(\phi^{i +1}(\bot)) = \phi^{i +2}(\bot)\,.
  \end{align*}
  Whence
  \begin{align*}
    \mu_{x}.\phi(x) &
    = \Nec[\bigwedge_{i = 1,\ldots, n} \alpha_{i}](\bigvee_{\beta \in
      \Side(\phi)} \beta) 
    = \Nec[\alpha_{n}]  \ldots \Nec[\alpha_{1}](\bigvee_{\beta \in
      \Side(\phi)} \beta) \leq \phi^{n + 1}(\bot)\,.
    \tag*{\qedhere}
  \end{align*}
\end{proof}

\begin{proposition}
  The above upper bound given in equation~\eqref{eq:convergencedf} is
  tight. 
\end{proposition}
\begin{proof}
  For each $n \geq 0$, consider the formula
  \begin{align*}
    \phi_{n}(x) & \eqdef b \vee \bigvee_{i=1,\ldots ,n} a_{i} \impl
    x\,,
  \end{align*}
  and the model $K_{n} = \langle P(\set{1,\ldots
    ,n}),\subseteq,b,\set{a_{i} \mid i = 1,\ldots ,n} \rangle$ with $b
  = \set{\emptyset}$, 
  $a_1, \ldots, a_n$ atomic formulas and  $s \forces a_{i}$ iff $i \not \in s$, for $s\in P(\set{1,\ldots
    ,n})$.  Let us
  compute the value of $\phi_{n}(x)$.
  \begin{align*}
    s \forces a_{i} \impl x & \tiff \forall s' \subseteq s, \;i \not
    \in s' \Rightarrow s' \forces x 
    \tiff s \setminus \set{i} \forces x\,,
    \intertext{whence}
    s \forces \phi_{n}(x) & \tiff  \text{either } s = \emptyset \text{
      or }  s \setminus \set{i} \forces x, \text{ for some } i \in \set{1,\ldots ,n}\,.
  \end{align*}
  Thus it is immediate to see that
  \begin{align*}
    \phi^{k+1}(\emptyset)
    & = \set{s \subseteq \set{1,\ldots ,n} \mid \card s \leq k}
  \end{align*}
  so that $\phi_{n}$ converges in no less than $n+1$ steps.
\end{proof}

\section{Ruitenburg's numbers for \fullypositive formulas}
\label{sec:ruitenburg}

Let $\phi$ be a formula of the \IPC (possibly) containing the variable
$x$. By $\phi^{n}$ we denote the iterated substitution of $x$ in
$\phi$ for $\phi$, defined by induction by $\phi^{0} \eqdef x$ and
$\phi^{n+1} \eqdef \phi[\phi^{n}/x]$. 
We let $\rho(\phi)$ be the least
non-negative integer $n$ such that the relation
$\phi^{n+2} = \phi^{n}$ holds; $\rho(\phi)$ is defined for any formula
$\phi$ of the \IPC, by \cite{Ruitenburg84}, and moreover
$\cl(\phi) \leq \rho(\phi)$.  A fine analysis of Ruitenburg's work
shows that $\rho(\phi) \leq 2n + 2$, where $n$ counts the implication
subformulas and the propositional variables in $\phi$.

The tools developed until now allow to construct an upper bound for
$\cl(\phi)$ for any formula $\phi$ of the \IPC, yet the bound so
obtained is exponential in the size of $\phi$; thus, in view of the
relation $\cl(\phi) \leq \rho(\phi) \leq 2n +2$, it is not
optimal. We exemplify this point.
Let $\phi$ be a \fullypositive formula and let $n$ be its size (the
number of all symbols and propositional variables in $\phi$).  When
transforming $\phi$ into a conjunction of disjunctive formulas, so
\begin{align}
  \label{eq:difficultConjunction}
  \phi & \eqIpc \bigwedge_{i = 1,\ldots ,k} \phi_{i}\,,
\end{align}
the number $k$ of conjuncts might be exponentially biggger than $n$.
Say that $\cl(\phi_{i}) \leq N$ for each $i = 1,\ldots ,k$.
An iterated applications of
Proposition~\ref{prop:conjunctionconvergence} yields the following upper bound:
\begin{align*}
  \cl(\phi) &
              = \cl(\bigwedge_{i} \phi_{i})
              \leq
              1 + \sum_{i = 1,\ldots ,k} (\cl(\phi_{i}) -1)
              \leq 1 + k(N -1)\,,
\end{align*}
which depends on some possibly very large $k$.

\bigskip

From now on, our goal shall be to give an upper bound for $\cl(\phi)$
when $\phi$ is a formula such as the one in (either side of) equation
\eqref{eq:difficultConjunction}. Since our proofs actually yield upper
bounds for Ruitenburg's numbers $\rho(\phi)$ (and a proof of
Ruitenburg's theorem for these formulas) we present our results
directly as bounds for the numbers $\rho(\phi)$.

While the procedure that transforms a \stronglypositive formula $\phi$
(say as the one on the left of \eqref{eq:difficultConjunction}) into a
conjunction of disjunctive formulas $\phi_{i}$ (as the one on the
right of \eqref{eq:difficultConjunction}) might exponentially increase
the size of the formula, as argued above, it does not increase the
number of head subformulas nor the number of side subformulas.
Therefore we give bounds as functions of these two parameters, which
eventually ensures an upper bound to Ruitenburg's numbers of
\stronglypositive formulas which is quadratic in the size of the
formulas.  
In view of obtaining these upper bounds we can (and shall) suppose
that all the head or side subformulas are propositional variables.

\medskip

In the following we let $\A \eqdef \set{\alpha_{1},\ldots ,\alpha_{N}}$
and $\B \eqdef \set{\beta_{1},\ldots ,\beta_{M}}$ be two (finite) disjoint
sets of propositional variables; we also suppose that the special
propositional variable $x$ does not belong to either of $\A$ and $\B$.
We consider formulas of the \IPC generated by the following grammar:
\begin{align}
  \label{grammar:disjunctiveAB}
  \phi \;& \production \;x  \;\mid\;\NEC{A} \phi \;\mid \;(\bigvee B) \vee \phi \;\mid\;
  \phi \vee \phi\,,
\end{align}
where $A \subseteq \A$, $B \subseteq \B$ and, as before,
$\NEC{A}\phi = \bigwedge A \impl \phi$.  That is, formulas generated by
the above grammar are disjunctive formulas, as defined by the grammar
\eqref{grammar:disjunctive}, whose head formulas are conjunctions of
propositional variables from $\A$, and whose side formulas are
disjunctions of propositional variables from $\B$.  We let $\DisjAB$
be the set of formulas generated by \eqref{grammar:disjunctiveAB}.
We consider formulas in $\DisjAB$ as elements of
$\FH[\alpha_{1},\ldots ,\alpha_{N},\beta_{1},\ldots ,\beta_{M},x]$,
the free Heyting algebra on the generators
$\alpha_{1},\ldots ,\alpha_{N},\beta_{1},\ldots
,\beta_{M},x$. Substitution of a formula $\psi$ for the variable $x$
in a formula $\phi$, usually noted by $\phi[\psi/x]$, yields a monoid
structure on
$\FH[\alpha_{1},\ldots ,\alpha_{N},\beta_{1},\ldots ,\beta_{M},x]$. We
write $\phi \circ \psi$ for $\phi[\psi/x]$ or sometimes,
$\phi(\psi)$. Since formulas in $\DisjAB$ are closed under
substitution, $\DisjAB$ is a submonoid of
$\FH[\alpha_{1},\ldots ,\alpha_{N},\beta_{1},\ldots
,\beta_{M},x]$. $\DisjAB$ is actually an ordered submonoid, meaning
that the following clause is vaild:
\begin{align}
  \label{eq:ordmonoid}
  \phi \leq \phi' \tand & \psi \leq \psi'  \timplies \phi \circ \psi \leq \phi' \circ \psi'\,.
\end{align}
This is mainly because the variable $x$ never occurs under the left
side of 
any implication in a formula in $\DisjAB$. Moreover, formulas
are inflating, meaning that
\begin{align}
  \label{eq:inflating}
  x &\leq \phi\,, \quad \text{for each $\phi \in \DisjAB$\,.}
\end{align}

\subsection{The support of a formula}
We define next two functions, $\suppA$ and $\suppB$, with domain
$\DisjAB$ and codomain $\PA$ and $\PB$, respectively: \\
\pptiny
\begin{minipage}[t]{0.49\linewidth}
\begin{align*}
  \suppA(x) & \eqdef \emptyset\,, \\
  \suppA(\,\NEC{A} \phi\,) & \eqdef A \cup \suppA(\phi)\,, \\
  \suppA(\,(\bigvee B) \vee \phi\,) & \eqdef \suppA(\phi) \,,\\
  \suppA(\phi_{0} \vee \phi_{1}) & \eqdef
  \suppA(\phi_{0}) \cup
  \suppA(\phi_{1})\,,
\end{align*}
\end{minipage}
\begin{minipage}[t]{0.49\linewidth}
\begin{align*}
  \suppB(x) & \eqdef \emptyset\,, \\
  \suppB(\,\NEC{A} \phi\,) & \eqdef \suppB(\phi)\,, \\
  \suppB(\,(\bigvee B) \vee \phi\,) & \eqdef B \cup \suppB(\phi) \,, \\
  \suppB(\phi_{0} \vee \phi_{1}) & \eqdef
  \suppB(\phi_{0}) \cup \suppB(\phi_{1})\,.
\end{align*}
\end{minipage}
\ppnormal
\smallskip
\\
We also let
\begin{align*}
  \supp(\phi) & \eqdef (\suppA(\phi),\suppB(\phi))\,,
\end{align*}
so $\supp(\phi) \in \PA \times \PB$.

\subsection{Word formulas}
In the inverse direction, given
$(A,B) \in \PA \times \PB$, we define
\begin{align*}
  \phi_{(A,B)} & \eqdef\NEC{A} (\bigvee B \vee x)\,.
\end{align*}
Let us develop the basic properties of the formulas $\phi_{(A,B)}$.
\begin{proposition}
  \label{prop:ABrelations}
  For each $(A_{0},B_{0}),(A_{1},B_{1}) \in \PABast$ and each $\phi \in \DisjAB$,
  \begin{align*}
    \phi_{(A_{0},\emptyset)} \circ \phi_{(A_{1},B_{1})}
    & =     \phi_{(A_{0} \cup A_{1},B_{1})}\,, \\
    \phi_{(A_{0},B_{0})} \circ \phi_{(\emptyset,B_{1})}
    & =     \phi_{(A_{0} ,B_{0}\cup B_{1})}\,, \\
    \phi_{(A_{0},B_{0})} \circ \phi \circ \phi_{(A_{1},B_{1})}
    & =  \phi_{(A_{0},B_{0}\setminus B_{1})} \circ \phi \circ \phi_{(A_{1}\setminus A_{0},B_{1})}\,.
  \end{align*}
\end{proposition}
\begin{proof}
  The first two properties are immediate from the definition of
  $\phi_{(A,B)}$. For the third, notice that
  \begin{align*}
    \phi_{(A_{0},B_{0})} \circ
    & \phi \circ \phi_{(A_{1},B_{1})}
    \\&
    =
    \NEC{A_{0} \cap A_{1}}
      \NEC{A_{0} \setminus A_{1}}
      (\bigvee B_{0} \vee \phi(\,
      \NEC{A_{0} \cap A_{1}}\NEC{A_{1} \setminus A_{0}}
      (\bigvee B_{1} \vee x)
      /x\,)
      )
    \\& =
      \NEC{A_{0} \cap A_{1}} \NEC{A_{0} \setminus A_{1}}
      (\bigvee B_{0} \vee \phi(\,
      \NEC{A_{1} \setminus A_{0}}
      (\bigvee B_{1} \vee x)
      /x\,)
      )
      \,,
    \tag*{by Lemma~\ref{lemma:idempotency},}
    \\&
    = \NEC{A_{0}} 
    (\bigvee B_{0} \vee \phi(\,
    \NEC{A_{1} \setminus A_{0}}
    (\bigvee B_{1} \vee x)
    /x\,)
    )
    \\&
    =
    \NEC{A_{0}}  
    (\bigvee B_{0}\setminus B_{1} \vee \phi(\,
    \NEC{A_{1} \setminus A_{0}}
    (\bigvee B_{1} \vee x)
    \,)
    )
    \,,
    \tag*{since  $\phi(\beta_{0} \vee x) = \beta_{0} \vee \phi(\beta_{0} \vee x)$,}
    \\&
    = \phi_{(A_{0},B_{0}\setminus B_{1})} \circ \phi \circ \phi_{(A_{1}\setminus A_{0},B_{1})}\,\,. \tag*{\qedhere}
  \end{align*}
\end{proof}
An immediate consequence of  the proposition is the following:
\begin{lemma}
  For each $(A,B) \in \PAB$, $\phi_{(A,B)}^{2} = \phi_{(A,B)}$, so
  $\rho(\phi_{(A,B)}) = 1$.
\end{lemma}
We extend the definition of the correspondence sending
$(A,B) \in \PA \times \PB$ to $\phi_{(A,B)} \in \DisjAB$ to the set of
all words over the alphabet $\PA \times \PB$---that shall be noted
by $(\PA \times \PB)^{\ast}$, as usual.  Syntactically, this amounts to defining
$\phi_{w}$ for each $w \in (\A\times \B)^{\ast}$, as follows:
\begin{align*}
  \phi_{\epsilon} & \eqdef x\,,
  & \phi_{(A,B)w} \eqdef \phi_{(A,B)} \circ \phi_{w}
    \,.
\end{align*}
We call a formula of the form $\phi_{w}$ for some $w \in \PABast$ a
\emph{word formula}.
\begin{lemma}
  For each $w \in (\PA \times \PB)^{\ast}$, $\phi_{w} \in \DisjAB$.
  Moreover, 
  if $w = (A_{1},B_{1}) \ldots (A_{k},B_{k})$, then 
  \begin{align}
    \label{eq:suppW}
    \supp(\phi_{w})  & = (\bigcup_{i = 1\ldots ,k} A_{i},\bigcup_{i =
                       1\ldots ,k}B_{i})\,, \\
    \label{eq:rhowiswto}
    \phi_{w} \leq \phi_{w}^{2} & = \phi_{\supp(\phi_{w})}\,,
    \\
    \rho(\phi_{w}) &\leq 2\,.
  \end{align}
\end{lemma}
\begin{proof}
  The first statement is a consequence of formulas of $\DisjAB$ being
  closed under substitution. Equation~\eqref{eq:suppW} is easily
  proved by induction. The relation $\phi_{w} \leq \phi_{w}^{2}$ is an
  easy consequence of conditions~\eqref{eq:ordmonoid} and
  \eqref{eq:inflating}.  $\phi_{w}^{2} = \phi_{\supp(\phi_{w})}$ is
  obtained by iteratively applying the relations in
  Proposition~\ref{prop:ABrelations}.  Finally we argue that
  $\phi_{w}^{3} = \phi_{w}^{2}$ (so $\rho(\phi_{w}) = 2$) as follows:
  \begin{align*}
    \phi_{w}^{2} & \leq  \phi_{w}^{3} \leq  \phi_{w}^{4}
                   = \phi_{\supp(\phi_{w})}^{2} = \phi_{\supp(\phi_{w})}  =  \phi_{w}^{2}\,.
                   \tag*{\qedhere}
  \end{align*}
\end{proof}
In view of \eqref{eq:suppW}, let us define
\begin{align*}
  \supp(\,(A_{1},B_{1}) \ldots (A_{k},B_{k})\,)
  & \eqdef 
    (\bigcup_{i = 1\ldots ,k}A_{i},\bigcup_{i =
    1\ldots ,k}B_{i})\,,
\end{align*}
so $\supp(w) = \supp(\phi_{w})$.

\begin{lemma}
  For each $\phi \in \DisjAB$, $\phi \leq \phi_{\supp(\phi)}$.
\end{lemma}
\begin{proof}
  We inductively define, for each $\phi \in \DisjAB$,  a word $w(\phi)$
  such that $\phi \leq \phi_{w(\phi)}$ and
  $\supp(\phi) = \supp(w(\phi))$.  Then, using
  equation~\eqref{eq:rhowiswto},
  we deduce
  \begin{align*}
  \phi & \leq   \phi_{w(\phi)} \leq \phi_{\supp(w(\phi))} = \phi_{\supp(\phi)}\,.
  \end{align*}
  
  We let $w(x) \eqdef (\emptyset,\emptyset)$,
  $w(\NEC{A}\phi) \eqdef (A,\emptyset)w(\phi)$,
  $w(\bigvee B \vee \phi) \eqdef (\emptyset,B)w(\phi)$ , and
    \begin{align*}
    w(\phi_{0} \vee \phi_{1}) & \eqdef  w(\phi_{0}) w(\phi_{1}) \,.
    \end{align*}
    By induction, it is proved that $\phi \leq \phi_{w(\phi)}$ and
    $\supp(\phi) = \supp(w(\phi))$, the only non-obvious inductive
    case being the last, which we prove next.  For $i = 0,1$, let
    $w_{i} \eqdef w(\phi_{i})$ and suppose that
    $\phi_{i} \leq \phi_{w_{i}}$ and
    $\supp(\phi_{i}) = \supp(\phi_{w_{i}})$.  Then
    $\supp(\phi) = \supp(\phi_{0}) \cup \supp(\phi_{1}) =
    \supp(\phi_{w_{0}}) \cup \supp(\phi_{w_{1}}) =
    \supp(\phi_{w_{0}}\circ\phi_{w_{1}}) = \supp(\phi_{w_{0}w_{1}})$
    and
  \begin{align*}
    \phi & = \phi_{0} \vee \phi_{1}
    \leq \phi_{0} \circ \phi_{1} \leq \phi_{w_{0}}\circ\phi_{w_{1}}
    = \phi_{w(\phi)}
    \,,
  \end{align*}
  where the relation
  $\phi_{0} \vee \phi_{1} \leq \phi_{0} \circ \phi_{1} $ is a
  consequence of $\phi_{i}$, $i = 0,1$, being inflating.
\end{proof}
We shall see later---as a particular instance of
Theorem~\ref{thm:mainRuitenburg}---that
$\phi^{n} = \phi_{\supp(\phi)}$ for some $n$, and for each
$\phi \in \DisjAB$. That is, $\phi_{\supp(\phi)}$ yields a closed
expression of the formula $\phi^{\rho(\phi)}$.
We shall further exploit word formulas in the rest of the section and
heavily rely on the next observation.
\begin{definition}
  For $(A,B) \in \PAB$ and
  $w =(A_1, B_1), \dots, (A_k, B_k) \in \PABast$, we let
\begin{align}
  \label{eq:less}
  (A, B) \Less (A_{1},B_{1}) \ldots (A_{k},B_{k}) \; \tiff\; \exists l\leq k
  \tst A\subseteq
  \bigcup_{j\leq l} A_j ~\tand~ B\subseteq \bigcup_{j\geq l} B_j\;.
\end{align}
\end{definition}
\begin{proposition}
  \label{prop:Less}
  If $(A,B)\Less w$, then $\phi_{(A, B)} \leq \phi_{w}$.
\end{proposition}
\begin{proof}
  Let $w=(A_1, B_1), \dots, (A_k, B_k)$ and let $\ell$ be such that
  \eqref{eq:less} holds. Define
  $w_{L} = (A_1, B_1), \dots, (A_\ell, \emptyset)$ and
  $w_{R} = (\emptyset, B_\ell), \dots, (A_n, B_k)$.
  Observe that 
  \begin{align*}
    \phi_{(A,\emptyset)}
    & \leq \phi_{(\bigcup_{i=1\ldots \ell} A_{i},\emptyset)}
      = \phi_{(A_{1},\emptyset)}\circ \ldots \circ
      \phi_{(A_{\ell},\emptyset)} 
    \leq  \phi_{(A_{1},B_{1})}\circ \ldots \circ
      \phi_{(A_{\ell-1},B_{\ell -1})} \circ \phi_{(A_{\ell},\emptyset)}=
      \phi_{w_{L}}\,,
  \end{align*}
  and, similarly, $\phi_{(\emptyset,B)} \leq \phi_{w_{R}}$.
  It follows that
  $\phi_{(A,B)} = \phi_{(A,\emptyset)} \circ \phi_{(\emptyset,B)} \leq
  \phi_{w_{L} }\circ \phi_{w_{R}} = \phi_{w}$.  
\end{proof}

\subsection{Conjunctions of \Star formulas}
In the next definition, if $X, Y \subseteq \PABast$, then we let
\begin{align*}
  X \cdot Y & \eqdef \set{wv \mid w \in X,\, v \in Y}\,.
\end{align*}
\begin{definition}
  The set $\Branches(\phi)$ of branches of $\phi \in \DisjAB$ is
  defined by induction:
  \begin{align*}
    \Branches(x) & \eqdef \set{\epsilon} \\
    \Branches(\bigvee B \vee \phi) & \eqdef \set{(\emptyset,B)}\cdot \Branches(\phi) \\
    \Branches(\NEC{A}\phi) & \eqdef \set{(A,\emptyset)}\cdot \Branches(\phi) \\
    \Branches(\phi_{0} \vee \phi_{1}) & \eqdef \Branches(\phi_{0}) \cup \Branches(\phi_{1})\,.
  \end{align*}
  The formula $\br(\phi)$ is then defined as follows:
  \begin{align*}
    \br(\phi) & \eqdef \bigvee \set{\phi_{w} \mid w \in \Branches(\phi) }\,. 
  \end{align*}
  A formula $\phi \in \DisjAB$ is a \emph{\Star formula} if
  $\br(\phi) = \phi$.
\end{definition}

\begin{lemma}
  For each $\phi \in \DisjAB$, 
  $\br(\phi) \leq \phi$ and $\supp(\br(\phi)) = \supp(\phi)$.
\end{lemma}
\begin{proof}
  A straightforward induction:
  \begin{align*}
    \br(x)
    & = \phi_{\epsilon} = x\,. \\
    \br((\bigvee B) \vee \phi) & = \bigvee \set{
      \phi_{(\emptyset,B)}\circ\phi_{w} \mid w \in\Branches(\phi)} \\
    &
    \leq \phi_{(\emptyset,B)}\circ (\bigvee \set{ \phi_{w} \mid w
      \in\Branches(\phi)}) \leq \phi_{(\emptyset,B)}\circ \phi =
    (\bigvee B) \vee \phi\,.
    \\
    \br(\NEC{A} \phi) & = \bigvee \set{
      \phi_{(A,\emptyset)}\circ\phi_{w} \mid w \in\Branches(\phi)} \\
    &
    \leq \phi_{(A,\emptyset)}\circ (\bigvee \set{ \phi_{w} \mid w
      \in\Branches(\phi)}) \leq \phi_{(A,\emptyset)}\circ \phi =
   \NEC{A} \phi\,. \\
    \br(\phi_{0} \vee \phi_{1})
    & = \bigvee \set{\phi_{w} \mid w \in \Br(\phi_{0})} \cup
    \set{\phi_{w} \mid w \in \Br(\phi_{1})} \\
    & = \br(\phi_{0}) \vee \br(\phi_{1}) \leq \phi_{0} \vee
    \phi_{1}\,.
      \tag*{\qedhere}
  \end{align*}
\end{proof}

We come back now to our original goal, that of estimating upper bounds
for formulas $\phi$ of the form $\phi = \bigwedge_{i \in I} \phi_{i}$
as in display~\eqref{eq:difficultConjunction},
where now $\phi_{i} \in \DisjAB$ for each $i \in I$.  The next
Proposition reduces the problem of giving a closed expression for
$\phi^{\rho(\phi)}$ and estimating an upper bound for the Ruitenburg
number of $\phi$ as in \eqref{eq:difficultConjunction} to that of a
conjunction of \Star formulas, that is, formulas of the form
\begin{align}
  \label{eq:simplifiedPhi}
  \phi & \eqdef \bigwedge_{i} \phi_{i}\,,
         \quad \text{with} \quad
         \phi_{i}  \eqdef \bigvee_{j  \in J_{i}} \phi_{w_{i,j}}\;\tand\; w_{i,j} \in \PABast\,.
\end{align}
To understand how we shall use
Proposition~\ref{prop:candidateClosedFormula}, recall that
$\supp(\phi_i)= \supp(\br(\phi_i))$ for all $i$; moreover, we shall
show (Propositions~\ref{prop:winningToClosedRuit}
and~\ref{prop:winningStrategy} below) that
$\bigwedge_i \phi_{\supp(\br(\phi_i))} \leq (\bigwedge_i
\br(\phi_i))^n$ for $n$ large enough. These two facts entail
$\bigwedge_i \phi_{\supp(\phi_i)} \leq \bigwedge_i \br(\phi_i)^n$ (for
large $n$), which is the condition under which
Proposition~\ref{prop:candidateClosedFormula} holds.

\begin{proposition}
  \label{prop:candidateClosedFormula}
  Let $I$ be a finite set, 
  let $\phi_i\in \DisjAB$ for each $i \in I$, and
  let $n \geq 0$; suppose that  
  $\bigwedge_{i} \phi_{\supp(\phi_{i})} \leq (\bigwedge_{i}
  \br(\phi_{i}))^{n}$. Then the following holds:
  \begin{enumerate}[(i)]
  \item $\bigwedge_{i} \phi_{\supp(\phi_{i})} \leq (\bigwedge_{i}
    \phi_{i})^{n}$,
  \item 
    $\bigwedge_{i} \phi_{\supp(\phi_{i})} = (\bigwedge_{i}
    \phi_{i})^{\rho(\bigwedge_{i} \phi_{i})}$, and
  \item
    $\rho(\bigwedge_{i} \phi_{i}) \leq \rho(\bigwedge_{i}
    \br(\phi_{i}))$.
  \end{enumerate}
\end{proposition}
\begin{proof}
  Statement (i) of the proposition follows from
  $\bigwedge_{i} \br(\phi_{i}) \leq \bigwedge_{i} \phi_{i}$, so
  $(\bigwedge_{i} \br(\phi_{i}))^{n} \leq (\bigwedge_{i}
  \phi_{i})^{n}$ and
  $\bigwedge_{i} \phi_{\supp(\phi_{i})} \leq (\bigwedge_{i}
  \br(\phi_{i}))^{n} \leq (\bigwedge_{i} \phi_{i})^{n}$.
  We observe now that the relation
  $\bigwedge_{i} \phi_{\supp(\phi_{i})} \leq (\bigwedge_{i}
  \phi_{i})^{n}$ implies 
  \begin{align*}
    \bigwedge_{i} \phi_{\supp(\phi_{i})} = (\bigwedge_{i}
    \phi_{i})^{n}\,.
  \end{align*}
  To this goal, it is enough to argue that
  $(\bigwedge_{i} \phi_{i})^{n} \leq \bigwedge_{i}
  \phi_{\supp(\phi_{i})}$, for each $n \geq 0$, which follows from
  $(\bigwedge_{i} \phi_{i})^{n} \leq \phi_{i}^{n} \leq
  \phi_{\supp(\phi_{i})}^{n} = \phi_{\supp(\phi_{i})}$ (since
  $ \phi_{\supp(\phi_{i})}$ is idempotent), for each $i \in I$.

  Therefore, if $(i)$ holds, then
  $\bigwedge_{i} \phi_{\supp(\phi_{i})} \leq (\bigwedge_{i}
  \phi_{i})^{n +1}$ as well, since
  $(\bigwedge_{i} \phi_{i})^{n} \leq (\bigwedge_{i} \phi_{i})^{n +1}$,
  and then
  \begin{align*}
    (\bigwedge_{i} \phi_{i})^{n+1} & =\bigwedge_{i}
    \phi_{\supp(\phi_{i})} = (\bigwedge_{i} \phi_{i})^{n}\,.
  \end{align*}
  From these relations we immediately infer (ii). For (iii) we argue
  as follows. Let $K_{0 }= \rho(\bigwedge_{i} \br(\phi_{i}))$ and
  $K_{1} = \rho(\bigwedge_{i} \phi_{i})$; since
  $\supp(\phi_{i}) = \supp(\br(\phi_{i}))$, we also derive
  $\bigwedge_{i} \phi_{\supp(\phi_{i})} = (\bigwedge_{i}
  \br(\phi_{i}))^{K_{0}}$ as an instance of (ii).  The relation
  $K_{1} \leq K_{0}$ follows then by the inequalities
  \begin{align*}
    (\bigwedge_{i} \br(\phi_{i}))^{k} & \leq (\bigwedge_{i}
    \phi_{i})^{k} \leq \bigwedge_{k}
    \supp(\phi_{i}) = (\bigwedge_{i}
    \phi_{i})^{K_{1}}  = (\bigwedge_{i}
    \br(\phi_{i}))^{K_{0}} \,,
  \end{align*}
  valid for any $k \geq 0$.
\end{proof}

Let us give an explicit form to the iterates of a formula $\phi$ as in
\eqref{eq:simplifiedPhi}. To this goal, we shall assume that
$J_{i} = \set{1,\ldots ,k} = [k]$ for each $i \in I$. We do not loose
generality with this assumption, since the formula $\phi_{i}$ is
equivalent to $\phi_{i} \vee \phi_{\epsilon}$. We shall make use of
the distributive law (of disjunctions w.r.t. conjuctions) in the
following form:
\begin{align}
  \label{eq:distr}
  \bigvee_{j \in [k]} \bigwedge_{i \in I} X_{j,i}
  & = \bigwedge_{f : [k] \xrightarrow{} I} \bigvee_{j \in [k]} X_{j,f(j)}\,.
\end{align}

Let us also introduce the following notation:
\begin{align*}
  \Str_{n} & \eqdef \prod_{1 \leq \ell \leq n} I^{[k]^{\ell-1}}\,.
\end{align*}
An element $f \in \Str_{n}$ is a tuple $(f_{1},\ldots ,f_{n})$ with
$f_{\ell} : [k]^{\ell -1} \rto I$, for each $\ell = 1,\ldots ,n$. In
particular, for $\ell = 1$, we identify
$f_{1} \in I^{[k]^{0}} \iso I^{[1]}$ with an element of $I$. We think
of a tuple $(f_{1},\ldots ,f_{n}) \in \Str_{n}$ as a memory aware
strategy for the first player of a two players game: the strategy
tells him how to incrementally choose a tuple
$(i_{1},\ldots ,i_{n}) \in I^{n}$ as a function of the opponent's
choices $(j_{1},\ldots j_{n-1})$ (where $j_{\ell}\in [k]$ for
$\ell = 1,\ldots ,n-1$), so
$i_{\ell} = f_{\ell}(j_{1},\ldots ,j_{\ell-1})$ for
$\ell = 1,\ldots ,n$.
We
recall that there is a canonical bijection between
$I \times \Str_{n}^{[k]}$ and $\Str_{n+1}$, as witnessed by the
following computations:
\begin{align*}
  I \times \Str_{n}^{[k]} & = I \times (\prod_{1 \leq \ell \leq n}
  I^{[k]^{\ell -1}})^{[k]}
  \iso I^{[k]^{0}}\times \prod_{1 \leq \ell \leq n} I^{[k]^{\ell}} \iso
  \Str_{n +1}\,.
\end{align*}
An explicit description of the bijection is as follows:
\begin{align*}
  (f_{0},g) \in I \times \Str_{n}^{[k]}
  & \mapsto (f_{0},h_{1},\ldots ,h_{n}) \in \Str_{n+1}\,,
    \intertext{where for $\ell \geq 1$ we have}
    h_{\ell}(j_{1},j_{2},\ldots ,j_{\ell-1}) & = [g(j_{1})]_{\ell}(j_{2},\ldots ,j_{\ell -1})\,.
\end{align*}

\begin{proposition}
  \label{prop:indexes}
  Let $\phi$ be of the form as in display~\eqref{eq:simplifiedPhi}.
  For each $n \geq 1$, we have
  \begin{align}
    \label{eq:formeNormale}
    \phi^{n} & = \bigwedge_{(f_{1},\ldots ,f_{n}) \in \Str_{n}} \bigvee_{j_{1}}
    \phi_{w_{f_{1},j_{1}}}(\bigvee_{j_{2}} \phi_{w_{f_{2}(j_{1}),j_{2}}}(
    \ldots \bigvee_{j_{n}} \phi_{w_{f_{n}(j_{1},j_{2}\ldots
      j_{n-1}),j_{n}}} ))\,.
\end{align}
\end{proposition}
In order to ease reading, we shall write in the proof of the
proposition and in the rest of this section $\phi_{i,j}$ in place of
the more appropriate $\phi_{w_{i,j}}$.
\begin{proof}[Proof of Proposition~\ref{prop:indexes}]
  When $n = 1$,  then
  equation~\eqref{eq:formeNormale} reduces to
  \begin{align*}
    \phi & = \bigwedge_{f_{1} \in I} \bigvee_{j_{1} \in [k]} \phi_{f_{1},j_{1}}\,,
  \end{align*}
  so it holds simply by definition of $\phi$.
  Notice now that a word formula, and in particular each $\phi_{i,j}$,
  commutes with conjunctions; we use this fact in the inductive step.
  We suppose that \eqref{eq:formeNormale} holds for $n \geq 1$ and
  compute as follows:
  \pptiny
  \begin{align*}
    \phi^{n+1} & = \phi(\phi^{n}) \\
    & = \bigwedge_{i_{0} \in I}\bigvee_{j_{0} \in [k]}\phi_{i_{0},j_{0}}(\bigwedge_{f
      \in \Str_{n}} \bigvee_{j_{1}} \phi_{f_{1},j_{1}}(\bigvee_{j_{2}}
    \phi_{f_{2}(j_{1}),j_{2}}( \ldots \bigvee_{j_{n}}
    \phi_{f_{n}(j_{1},j_{2}\ldots
      j_{n-1}),j_{n}} )))\,,
    \tag*{by the inductive hypothesis,}
    \\
    & = \bigwedge_{i_{0} \in I}\bigvee_{j_{0} \in [k]}\bigwedge_{f
      \in \Str_{n}} \phi_{i_{0},j_{0}}(\bigvee_{j_{1}} \phi_{f_{1},j_{1}}(\bigvee_{j_{2}}
    \phi_{f_{2}(j_{1}),j_{2}}( \ldots \bigvee_{j_{n}}
    \phi_{f_{n}(j_{1},j_{2}\ldots
      j_{n-1}),j_{n}} )))\,,
    \tag*{since
      $\phi_{i_{0},j_{0}}$ commutes with conjunctions,}
    \\
    & = \bigwedge_{i_{0} \in I} \bigwedge_{g : [k] \rto[]
      \Str_{n}}\bigvee_{j_{0} \in [k]}\phi_{i_{0},j_{0}}(\bigvee_{j_{1}}
    \phi_{g(j_{0})_{1},j_{1}}(\bigvee_{j_{2}} \phi_{g(j_{0})_{2}(j_{1}),j_{2}}(\,\ldots 
    \ldots \bigvee_{j_{n}} \phi_{g(j_{0})_{n}(j_{1},j_{2}\ldots
      j_{n-1}),j_{n}} )))\,,
    \tag*{using~\eqref{eq:distr},}
    \\
    & = \bigwedge_{i_{0} \in I} \bigwedge_{g : [k] \rto[]
      \Str_{n}}\bigvee_{j_{1}}\phi_{i_{0},j_{1}}(\bigvee_{j_{2}}
    \phi_{g(j_{1})_{1},j_{2}}(\bigvee_{j_{3}} \phi_{g(j_{1})_{2}(j_{2}),j_{3}}(\,\ldots 
    \ldots \bigvee_{j_{n+1}} \phi_{g(j_{1})_{n}(j_{2},j_{3}\ldots
      j_{n}),j_{n+1}} ))) \\
    & = \bigwedge_{h \in
      \Str_{n +1}}
    \bigvee_{j_{1}}\phi_{h_{1},j_{1}}(\bigvee_{j_{2}}
    \phi_{h_{2}(j_{1}),j_{2}}(\bigvee_{j_{3}}
    \phi_{h_{3}(j_{1},j_{2}),j_{3}}( \,\ldots
    \ldots \bigvee_{j_{n+1}} \phi_{h_{n+1}(j_{1},j_{2},j_{3}\ldots
      j_{n}),j_{n+1}} ))) \,.
    \tag*{\qedhere}
  \end{align*}
  \ppnormal
\end{proof}

\subsection{A game for iterated conjunctions of \Star formulas}
 Let $\phi = \bigwedge_{i \in I} \phi_{i}$ with
$\phi_{i} =\bigvee_{j \in [k]} \phi_{w_{i,j}}$. For each $K \geq 1$,
we describe next a two-player game $\Gphi$ (between Eve and Adam, and
where Adam is the first player) with the following property: if Eve
has a winning straetgy in $\Gphi$, then the relation
\begin{align*}
  \bigwedge_{i \in I} \phi_{\supp(\phi_{i})} & \leq \phi^{K}
\end{align*}
holds. Therefore, using Proposition~\ref{prop:candidateClosedFormula},
if Eve has a winning straetgy in $\Gphi$, then
$\phi^{\rho(\phi)} = \bigwedge_{i} \phi_{\supp(\phi_{i})}$ and that
$\rho(\phi) \leq K$.
Positions and moves of the game $\Gphi$ are as follows.  Adam's
positions are of the form $(i_{1},j_{1}) \ldots (i_{n},j_{n})$, where
$n \leq K$ and, for $\ell = 1,\ldots ,n$, $i_{\ell} \in I$ and
$j_{\ell} \in [k]$. In such a position (when $n < K$) Adam chooses
$i_{n+1} \in I$ and moves to the position
$(i_{1},j_{1}) \ldots (i_{n},j_{n})(i_{n+1},?)$. In this position Eve
chooses $j_{n+1}$ and moves to
$(i_{1},j_{1}) \ldots (i_{n},j_{n})(i_{n+1},j_{n+1})$.  The length of
a position $(i_{1},j_{1}) \ldots (i_{n},j_{n})$ is the integer $n$.
The initial position is $\epsilon$ (the empty sequence or, in other
words, the sequence of length $n = 0$).

To each of Adam's position $p = (i_{1},j_{1}) \ldots (i_{k},j_{k})$,
let $w_{p}= w_{i_{1},j_{1}}\ldots w_{i_{k},j_{k}}$.  A terminal
position $p = (i_{1},j_{1}) \ldots (i_{K},j_{K})$ is a win for Eve
(and a loss for Adam) if, for some $i \in I$,
$\supp(\phi_{i}) \leq \phi_{w}$ with
$w = w_{i_{1},j_{1}}\ldots w_{i_{K},j_{K}}$.

\begin{proposition}
  \label{prop:winningToClosedRuit}
  If Eve has a winning strategy in the game $\Gphi$, then
  \begin{align*}
    \bigwedge_{i \in I} \supp(\phi_{i}) & \leq \phi^{K}\,.
  \end{align*}
\end{proposition}
\begin{proof}
  In view of \eqref{eq:formeNormale} we need to show that, for
  any $f \in \Str_{K}$,
  $\bigwedge_{i \in I} \supp(\phi_{i}) \forces \oplus({f})$, where
  \begin{align*}
    \oplus({f}) & \eqdef
    \bigvee_{j_{1}}
    \phi_{f_{1},j_{1}}(\bigvee_{j_{2}} \phi_{f_{2}(j_{1}),j_{2}}(
    \ldots \bigvee_{j_{K}} \phi_{f_{K}(j_{1},j_{2}\ldots
      j_{K-1}),j_{K}} ))\,.
  \end{align*}
  Let $f \in \Str_{K}$ be fixed and observe that such an $f$ yields a
  strategy (not a winning one) for Adam in the game $\Gphi$.  Now, if
  $g$ is a winning strategy for Eve in this game, then $f$ and $g$
  determine a play $f \vs g$ in the game such that, for some
  $i \in I$, $\phi_{\supp(\phi_{i})} \leq \phi_{w_{f \vs g}}$.  We
  have then $\bigwedge_{i} \phi_{\supp(\phi_{i})}
  \leq \phi_{\supp(\phi_{i})} \leq \phi_{w_{f \vs g}} \leq
  \oplus({f})$.
\end{proof}

Recall that $\card(\A) = N$ and $\card(\B) = M$. 
\begin{proposition} Eve has a winning strategy in the game
  $\Gphi[(N+1)(M +1)]$.
  \label{prop:winningStrategy}
\end{proposition}
\begin{proof}
  Eve keeps a memory in order to decide how to move. Her memory is a
  pair $(A,B) \in \PAB$ and, at the beginning of the play,
  $(A,B) = (\emptyset,\emptyset)$.
  
  At each position $p = (i_{1},j_{1}) \ldots (i_{\ell},j_{\ell})$ of the
  play, if the memory is $(A_{p},B_{p})$, then
  $A_{p} = \suppA(w_{p})$, where
  $w_{p} = w_{i_{1},j_{1}}w_{i_{2},j_{2}}\ldots w_{i_{\ell},j_{\ell}}$.  In
  particular, if $p'$ is a prefix of $p$, then
  $A_{p'} \subseteq A_{p}$.  Moreover, if $w_{p} = w_{0}w_{1}$ with
  $w_{0}$ being the shortest prefix of $w_{p}$ such that
  $\suppA(w_{0}) =A_{p}$, then $B_{p} \subseteq \suppB(w_{1})$.
  Notice that these conditions imply that $(A_{p},B_{p}) \Less w_{p}$,
  so $\phi_{(A_{p},B_{p})} \leq \phi_{w_{p}}$ by
  Lemma~\ref{prop:Less}.

  Let
  $p = (i_{1},j_{1}) \ldots (i_{\ell},j_{\ell})$. At position
  $p (i_{\ell+1},?)$, Eve chooses $j_{\ell+1}$ so that, if
  $p' = p(i_{\ell+1},j_{\ell+1})$, $\suppA(w_{p'})$ is stricly greater than
  $A_{p} = \suppA(w_{p})$. If it is possible to choose such $j_{\ell+1}$,
  then she updates her memory to $(\supp_{A}(w_{p'}),\emptyset)$.
  Otherwise, if it is not possible to choose $j_{\ell+1}$ with these
  properties, then Eve chooses $j_{\ell+1}$ so
  $\suppB(w_{i_{\ell+1},j_{\ell+1}})$ strictly includes $B_{p}$. She updates
  then the memory to
  $(A_{p'},B_{p'}) = (A_{p},B_{p} \cup
  \suppB(\phi_{w_{i_{\ell+1},j_{\ell+1}}}))$.
  If it is not possible to operate that kind of choices, then Eve chooses some
  $j_{\ell+1}$ at random and sets $(A_{p'},B_{p'}) = (A_{p},B_{p})$.

  Now, in a play, there are at most $N +1$ values for $A_{p}$ and, for
  each fixed $A_{p}$, there are at most $M+1$ values for
  $B_{p}$. Therefore, in $(N+1)(M +1)$ rounds either
  \begin{itemize}
  \item[(a)] the play visits an Eve's position
    $p (i_{\ell+1},?)$---therefore with $\ell < (N+1)(M+1)$ and $p$ of
    the form $(i_{1},j_{1}) \ldots (i_{\ell},j_{\ell})$---where she
    cannot extend $A_{p}$ nor $B_{p}$;
    that is, we have $\suppA(w_{i_{\ell+1},j}) \subseteq A_{p}$ and
    $\suppB(w_{i_{\ell+1},j}) \subseteq B_{p}$, for each $j\in [k]$;
    or 
  \item[(b)] the play ends up in an Adam's position
    $p = (i_{1},j_{1}) \ldots (i_{\ell},j_{\ell})$ with $\ell = (N+1)(M+1)$,
    where now $A_{p} = \A$ and $B_{p} = \B$.
  \end{itemize}
  Suppose (a).  Since
  $\supp(\phi_{i_{\ell +1}}) = (\bigcup_{j}
  \suppA(w_{i_{\ell+1},j}),\bigcup_{j} \suppA(w_{i_{\ell+1},j}))$, we have
  $\suppA(\phi_{i_{\ell +1}}) \subseteq A_{p}$ and
  $\suppB(\phi_{i_{\ell +1}}) \subseteq B_{p}$. Since
  $(A_{p},B_{p}) \Less w_{p}$, it also follows that
  $\supp(\phi_{i_{\ell +1}}) \Less w_{p}$, so
  $\phi_{\supp(\phi_{i_{\ell+1}})} \leq \phi_{w_{p}}$ by
  Lemma~\ref{prop:Less}.  This shows that the position $p$ (as well as
  any of its extensions) is a win for Eve.  If (b) then $A_{p} = \A$
  and $B_{p} = \B$ so, in a similar way as before, we have
  $\phi_{\supp(\phi_{i})} \leq \phi_{w_{p}}$, this time for each
  $i \in I$.
\end{proof}

We can now state the main result of this section.
\begin{thm}
  \label{thm:mainRuitenburg}
  Let $\phi = \bigwedge_{i \in I} \phi_{i}$ where each $\phi_{i}$ is a
  disjunctive formula. Then $\rho(\phi) \leq (N+1)(M+1)$ where $N$ is
  the number of distinct head subformulas of $\phi$ and $M$ is the
  number of distinct side subformulas occurring in any of  the $\phi_{i}$.
\end{thm}
\begin{proof}
  The statement holds iff and only if it holds when head and side
  subformulas of $\phi$ are propositional variables, that is, when
  $\phi_{i} \in \DisjAB$ for each $i \in I$, with $\card(\A) = N$ and
  $\card(\B) = M$. Moreover, according to
  Proposition~\ref{prop:candidateClosedFormula}, the statement of the
  theorem holds if
  $\bigwedge_{i \in I} \phi_{\supp(\phi_{i})} \leq \phi^{(N+1)(M+1)}$
  and under the additional assumption that each
  $\phi_{i}$ is a \Star formula.
  Now the relation
  $\bigwedge_{i \in I} \phi_{\supp(\phi_{i})} \leq \phi^{(N+1)(M+1)}$
  is a consequence of Proposition~\ref{prop:winningStrategy}, stating
  that Eve has a winning strategy in the game $\Gphi[(N+1)(M +1)]$,
  and of Proposition~\ref{prop:winningToClosedRuit}, relating such a
  winning strategy to the relation.
\end{proof}

\begin{remark}
  The upper bound given in Theorem~\ref{thm:mainRuitenburg} appears to
  be orthogonal to bound implicit in Ruitenburg's paper
  \cite{Ruitenburg84}. In the bound $\rho(\phi) \leq 2n +2$, the size
  $n$ of $\phi$ is at least the number of implication subformulas of
  $\phi$. Now, in a formula of the form $\bigwedge_{i \in I} \phi_{i}$
  with $\phi_{i} \in \DisjAB$, the number of implication subformulas
  might be exponentially larger than $N$ and $M$. Therefore the bound
  given in Theorem~\ref{thm:mainRuitenburg} is in this case tighter
  than Ruitenburg's bound.
  Conversely, we can derive from Theorem~\ref{thm:mainRuitenburg} a
  quadratic (in the size of the formula) upper bound for Ruitenburg's
  numbers of \stronglypositive formulas. This is achieved by
  considering that the size of a \stronglypositive formula is greater
  than the number of all the head and side subformulas in the
  conjuncts of its normal form (as in Lemma~\ref{lemma:CNF}).
  Ruitenburg's upper bound is in this case tighter.
\end{remark}

\begin{remark}
  The following example shows that the quadratic upper bound is
  necessary, at least with respect to finding a winning strategy for
  Eve.
  Let $\A \eqdef \set{\alpha_{1},\ldots ,\alpha_{N}}$ and
  $\B \eqdef \set{\beta_{1},\ldots ,\beta_{N}}$.  For each $k = 1,\ldots ,N$,
  let $P_{k}(\B)$ be the set of subsets of $\B$ if size $k$.
  Let $I = \set{(k,B) \mid B \in P_{k}(\B)}$ and, for each
  $(k,B) \in I$, consider the branch formula
  \begin{align*}
    \phi_{(k,B)} \eqdef & \bigvee_{\beta \in B}
    \phi_{(\set{\alpha_{k}},\set{\beta})}\,.
  \end{align*}
  Adam can use the following winning strategy in all the games
  $\Gphi[K]$ with $K < \frac{N(N-1)}{2}$.  He starts by choosing
  $(N,\B)$ until Eve has chosen at least $N-1$ different symbols from
  $\B$.  Let $\beta_{N}$ the only symbol not chosen by Eve. Then Adam
  chooses $(N-1,\B\setminus \set{\beta_{N}})$ and iterates this choice
  until Eve has chosen exactly $N-2$ different symbols. Let
  $\beta_{N -1}$ be the only symbol from
  $\B \setminus \set{\beta_{N}}$ which has not been chosen by Eve,
  then Adam chooses $(N-2,\B \setminus \set{\beta_{N},\beta_{N-1}}$,
  and so on.  Eve needs $N-1+(N-2)+(N-2)+ \ldots $ rounds to win.
  This example raises the question of the completeness of the game:
  does the existence of an Adam's winning strategy in $\Gphi$ implies
  that
  $\bigwedge_{i \in I} \phi_{\supp(\phi_{i})} \not \leq \phi^{K}$?
\end{remark}

\begin{remark}
  \label{rem:clstriclylessrho}
  We considered 
  \begin{align*}
    \phi_{n}(x) & \eqdef \bigvee_{i = 1,\ldots ,n} \alpha_{i} \impl (\beta_{i}
    \vee x)\,.
  \end{align*}
  and used \fCube \cite{fCube} to compute the values of
  $\cl(\phi_{n})$ and $\rho(\phi_{n})$. For $n \in \set{2,3,4,5}$, we
  obtained that $\cl(\phi_{n}) = \rho(\phi_{n}) = n +1$.  This raises
  the question whether there is any \fullypositive formula of the \Ipc
  for which we have $\cl(\phi_{n}) < \rho(\phi_{n})$.
\end{remark}

\section{A constant upper bound for disjunctions of \Atops}
\label{sec:almostclosure}

In this \Section we exemplify how investigating (lower bounds of)
closure ordinals might lead to uncover non-trivial properties of
\Ha{s}.
Example~\ref{ex:binaryjoin} illustrated the elimination procedure in
the case of \weaklynegative \fterms. It considered 
a \fterm of the form
\begin{align*}
  \phi(x) & \eqdef \bigvee_{i \in I} (x \impl b_{i}) \impl a_{i}\,,
\end{align*}
where the index set was a two element set.
In view of the similarity of these formulas with the disjunctive
formulas of \Section~\ref{sec:procedure}, we conjectured that closure
ordinals of formulas as the ones above increase as the size of $I$
becomes larger---so to exhibit tightness of the upper bound on closure
ordinals of \weaklynegative \fterms presented in
\Proposition~\ref{prop:convergesforwefs}. Yet, all our automatized
tests, for which we used the tool \fCube \cite{fCube}, pointed towards
the opposite direction. We finally managed to disprove the conjecture:
all these \fterms converge to their \lfp{s} in $3$ steps.

\bigskip

Let $H$ be a \Ha.  For $a,b \in H$, we call $j_{a,b}$ defined by
\begin{align*}
  j_{a,b}(x) & \eqdef (x \impl a) \impl b\,,
\end{align*}
an \emph{\Acop} (briefly, an \emph{\acop}).  The reason is the
following: when $a = b$, then $j_{a,a}$ is a closure operator (that
is, a monotone inflating idempotent function on $H$); more than that,
it is a \emph{Lawvere-Tierney topology} or \emph{nucleus}, see
\cite[Chapter II, section 2]{Johnstone82}, meaning that they are
strong: $x \land j_{a,a}(y) \leq j_{a,a}(x \land y)$, for each
$x,y \in H$.
We shall consider disjunctions of \acop{s}, for which we need a
convenient notation: for a family of pairs
$\Pi = \set{(a_{i},b_{i}) \mid i \in I}$, we shall write
\begin{align}
  \label{eq:defphipi}
  \phiPi(x) & \eqdef \bigvee_{i \in I} j_{a_{i},b_{i}}(x)\,.
\end{align}

\subsection{Elementary properties of \atop{s}}

\begin{lemma}
  The following holds, for each $x \in H$:
  \label{lemma:increasing}
  \label{cor:increasing}
  \label{lemma:condidempotent}
  \begin{enumerate}[(i)]
  \item $x \leq j_{a,b}(x)$ if and only if $x \leq a \impl b$. In
    particular, if $a \leq b$, then $x \leq j_{a,b}(x)$;
  \item $j_{a,b}^{2}(x) = j_{a,b}(x)$ if and only
    $j_{a,b}(b) \leq j_{a,b}(x)$.  In particular, this holds when $b \leq x$.
\end{enumerate}
Consequently, the restriction of $j_{a,b}$ to the interval
$[b,a\impl b]$ is a closure operator.
\end{lemma}
\begin{proof}
  (i) $x \leq (x \impl a) \impl b$ iff $x \land x \impl a \leq b $, iff
  $x \land a \leq b $, iff $x \leq a \impl b$.
  For the second statement, notice that if $a \leq b$, then
  $x\leq \top = a \impl b$.

  (ii) Notice firstly that the condition $j^{2}_{a,b}(x) = j_{a,b}(x)$
  is equivalent to $j^{2}_{a,b}(x) \leq j_{a,b}(x)$. As a matter of
  fact, $j_{a,b}(x) \leq a \impl b$ for each $x \in H$ so by (i) we
  always have $j_{a,b}(x) \leq j_{a,b}^{2}(x)$.

  Therefore we prove that  $j^{2}_{a,b}(x) \leq j_{a,b}(x)$ is
  equivalent to $j_{a,b}(b) \leq j_{a,b}(x)$.
  By repeated use of compatibility, we have the following equality:
  \pptiny
  \begin{align*}
    j^{2}_{a,b}(x) \land (x \impl a) = j_{a,b}((x \impl a) \impl b)
    \land (x \impl a) & = j_{a,b}(((x \impl a) \impl b) \land (x \impl
    a)) \land (x \impl a) \,,
    \\
    & = j_{a,b}(b \land (x \impl a))  \land (x
    \impl a)
    =  j_{a,b}(b)  \land (x
    \impl a)\,.
  \end{align*}
  \ppnormal
  It follows that
\ppfootnotesize
\begin{align*}
  j^{2}_{a,b}(x) \leq j_{a,b}(x) & \quad\tiff\quad j_{a,b}(b) \land (x \impl
  a) = j^{2}_{a,b}(x) \land (x \impl a) \leq b \quad \tiff \quad j_{a,b}(b)
  \leq j_{a,b}(x)\,.  
\end{align*}
\ppnormal
Finally, if $b \leq x$, then $j_{a,b}(b) \leq j_{a,b}(x)$ so
  $j_{a,b}^{2}(x) = j_{a,b}(x)$.
\end{proof}
Since $j_{a,b}(\bot) = b$, $j_{a,b}(\top) = a \impl b$, and $j_{a,b}$
is monotone, we also remark:
\begin{lemma}
  \label{lemma:topjab}
  The image of $H$ via $j_{a,b}$ is contained in the interval
  $[b,a\impl b]$.
\end{lemma}
Thus we have $j_{a,b}(x) = j_{a,b}(x) \land (a\impl b)$. We shall
exploit this fact many times, in conjunction with strongness.  The
following Lemma exemplifies this.

\begin{lemma}
  \label{lemma:guardedInflating}
 If $f : H \rto H$ is a strong monotone mapping, then
  $j_{a,b}(f(x)) \leq j_{a,b}(f(j_{a,b}(x)))$.
\end{lemma}
\begin{proof}
  We compute as follows:
  \begin{align*}
    j_{a,b}(f(x)) & = j_{a,b}(f(x)) \land (a \impl b)\,, \tag*{by
      Lemma~\ref{lemma:topjab},}
    \\
    & = j_{a,b}(f(x \land (a \impl b))) \land (a \impl b), \tag*{since
      $j_{a,b}\circ f$ is strong,}
    \\
    & \leq j_{a,b}(f(j_{a,b}(x \land (a \impl b)))) \land (a \impl
    b)\,, \tag*{using Lemma~\ref{lemma:increasing}.(i) and the fact that
      $x \land
      (a\impl b) \leq (a\impl b)$,} \\
    & = j_{a,b}(f(j_{a,b}(x)))\,,
  \end{align*}
  where in the last step we have again used Lemma~\ref{lemma:topjab}
  and that fact that $j_{a,b}\circ f$ is strong.
\end{proof}

To end this \Section, it is useful to pinpoint two identities that
shall be useful later. The first one is obtained by repeatedly using
compatibility of $j_{a,b}$:
\begin{align*}
  j_{a,b}(x) \land c & = j_{a\land c,b\land c}(x \land c) \land c =
  j_{a\land c,b\land c}(x) \land c\,.
\end{align*}
In particular, since  $j_{a,b}(x) =   j_{a,b}(x) \land (a \impl b)$,
we derive
\begin{align}
  \label{eq:alwaysInflating}
   j_{a,b}(x) & = j_{a\land b,b}(x) \land (a \impl b)\,.
\end{align}
The latter identity relates a general \atop to a specific \atop
$j_{a,b}$ with the property that $a \leq b$ which---according to
Lemma~\ref{lemma:increasing}.(i)---is always inflating.

\subsection{Closure of \pfp{s} of strong monotone mappings under
  exponentiation}

The following Lemma asserts that \pfp{s} of strong monotone mappings
are closed under exponentiation. This property seems to be the hidden
principal ingredient in the proof of the main result of this section,
Theorem~\ref{thm:atops}.
\begin{lemma}
  \label{lemma:strangelyclosed}
  Let $g : H \rto H$ be a strong monotone mapping.  If
  $c \in \Pref_{g}$, then $x \impl c \in \Pref_{g}$, for each
  $x,c \in H$.
\end{lemma}
\begin{proof}
  The Lemma is an immediate consequence of
  equation~\eqref{eq:strongimpl}: $g(x \impl c)\leq x \impl g(c) \leq
  x \impl c$, when $c \in \Pref_{g}$.
\end{proof}

We shall study next when $j_{a,b}(x) = j_{a,c}(x)$.  Indeed, in view
of Lemma~\ref{lemma:strangelyclosed}, we shall have that $j_{a,b}(x)$
is a prefixed point of a strong $g$, if this equality holds and $c$ is
a \pfp of $g$.
\begin{lemma}
  \label{lemma:middle}
  We have $j_{a,b}(x) = j_{a,c}(x)$ if and only if
  $x \impl a \leq b \leftrightarrow c$.
  In particular, if $b \leq c \leq x \leq a \impl b$, then
  $j_{a,b}(x) = j_{a,c}(x)$.
\end{lemma}
\begin{proof}
  By symmetry, it will be enough to prove that
  $j_{a,b}(x) \leq j_{a,c}(x)$ if and only if
  $x \impl a \leq b \impl c$.

  Suppose that $x \impl a \leq b \impl c$. Then
  \begin{align*}
    ((x \impl a) \impl b) \land (x \impl a)
    & = b \land (x \impl a) \leq b \land (b \impl c)
    \leq c\,
  \end{align*}
  so $j_{a,b}(x) \leq j_{a,c}(x)$. 
  Conversely, suppose that $j_{a,b}(x) \leq j_{a,c}(x)$. Then
  \begin{align*}
     b \land (x \impl a) & =  ((x \impl a) \impl b) \land (x \impl a)
     \leq c \,,
  \end{align*}
  so $x \impl a \leq b \impl c$.

  For the last sentence, we can use the characterization we have just
  given.  Suppose $b \leq c \leq x \leq a \impl b$.  Then
  $x \impl a \leq \top = b \impl c$. Also
  $c \land x \impl a \leq x \land x \impl a = x \land a \leq a \impl b
  \land a \leq b$, so $x \impl a \leq c \impl b$.
\end{proof}

\begin{proposition}
  \label{prop:presprefixpoint}
  Let $g$ be a strong monotone mapping.  If
  $c \in \Pref_{g} \cap [f, e \impl f]$, then
  $j_{e,f}(x) \in \Pref_{g}$ for each $x \in [c, e \impl f]$.
\end{proposition}
\begin{proof}
  By the previous Collary, we can write $j_{e,f}(x) = j_{e,c}(x)$.  It
  follows then from Lemma~\ref{lemma:strangelyclosed} that
  $j_{e,f}(x) = j_{e,c}(x) \in \Pref_{g}$.
\end{proof}

\subsection{Convergence in $3$ steps for disjunctions of \atop{s}}
\newcommand{\aibi}{(a_{i} \impl b_{i})}

Let therefore $\Pi = \set{(a_{i},b_{i}) \mid i \in I}$ be fixed;
in order to improve readability, let us
put, for each $i\in I$,
\begin{align*}
  j_{i}(x) & \eqdef j_{a_{i},b_{i}}(x)\,.
\end{align*}

\begin{thm}
  \label{thm:atops}
  The fuction $\phiPi$ defined as in equation~\eqref{eq:defphipi}
  converges to its \lfp in $3$ steps.
\end{thm}
\begin{proof}
  We need to prove that
  $j_{k}(\phiPi^{3}(\bot)) \leq \phiPi^{3}(\bot)$, for each $k \in
  I$. If we put $b \eqdef \phiPi(\bot) = \bigvee_{i \in I} b_{i}$ then
  we need to show that
  \begin{align}
    \label{eq:maingoal}
    j_{k}(\phiPi^{2}(b)) & \leq \phiPi^{2}(b) \qquad\text{for each $k \in I$.}
  \end{align}
  Let, from now on, $k \in I$ be fixed and put
  \begin{align*}
    J_{k}(x) & \eqdef j_{a_{k} \land b_{k},b_{k}}(x)\,,
  \end{align*}
  so $j_{k}(x) = J_{k}(x) \land (a_{k} \impl b_{k})$ as from
  equation~\ref{eq:alwaysInflating}.
  We shall argue that, for
  each $i \in I$, the following relation holds:
  \begin{align}
    \label{eq:ik}
    j_{i}(J_{k}(\phiPi(b))) & \leq J_{k}(\phiPi(b))\,.
  \end{align}
  Once equation~\eqref{eq:ik} is proved, we
  prove~\eqref{eq:maingoal} 
  as follows:
  \begin{align*}
    j_{k}(\phiPi^{2}(b)) & = j_{k}(\,\bigvee_{i \in I}
    j_{i}(\phiPi(b))\,)
    \leq j_{k}(\,\bigvee_{i \in I} j_{i}(J_{k}(\phiPi(b)))\,)\,,
    \tag*{by Lemma~\ref{lemma:guardedInflating},}
    \\
    & \leq j_{k}(\,\bigvee_{i \in I} J_{k}(\phiPi(b))\,)
    = j_{k}( J_{k}(\phiPi(b)))
    \,,
    \tag*{using equation~\eqref{eq:ik},}
    \\
    & =  j_{k}( j_{k}(\phiPi(b)))
     \,,
    \tag*{since of $j_{k}(x) = j_{k}(x) \land (a_{k} \impl b_{k})$ and
      $j_{k}$ is strong,}
    \\
     &
    = j_{k}(\phiPi(b)) \,, \tag*{using
      $b_{k} \leq b \leq \phiPi(b)$ and Lemma~\ref{lemma:condidempotent}.(ii),}
    \\
    & \leq \phiPi^{2}(b) 
    \,.
  \end{align*}
  In order to prove that equation~\eqref{eq:ik} holds, we use
  Proposition~\ref{prop:presprefixpoint} and argue that a certain
  $j_{e,f}(x)$ is a \pfp of $j_{i}$.  Let, in the statement of the
  Proposition,
  \pptiny
  \begin{align*}
    e &\eqdef a_{k} \land b_{k} \land \aibi\,, 
    \quad f \eqdef b_{k} \land \aibi \,, 
    \quad c \eqdef j_{i}(b)\,, 
    \quad x \eqdef \phiPi(b) \land \aibi\,, 
    \quad
    g = j_{i}\,.
  \end{align*}
  \ppnormal To apply the Proposition, we need to verify that $(i)$
  $f \leq c \leq x \leq e \impl f$ and that $(ii)$ $c$ is a \pfp of
  $j_{i}$.
  \begin{myList}
  \item[$(i)$] We have
    $b \land a_{i} \impl b_{i} \leq a_{i} \impl b_{i}$ and therefore,
    by Lemma~\ref{lemma:increasing}.(i),
    \begin{align*}
      b \land \aibi & \leq j_{i}(b \land \aibi) = j_{i}(b) \land \aibi
      = j_{i}(b)\,.
    \end{align*}
    Using this relation, we see that
    \begin{align*}
      f = b_{k} \land \aibi & \leq b \land \aibi \\
      & \leq j_{i}(b) = c \\
      & \leq \phiPi(b) \land \aibi = x \\
      & \leq \top = e \impl f\,.
    \end{align*}
  \item[$(ii)$] From $b_{i} \leq b$ and
    Lemma~\ref{lemma:condidempotent}.(ii) it immediately follows that
    $c = j_{i}(b)$ is a \pfp of $j_{i}$.
  \end{myList}
  From $(i)$, $(ii)$ and Proposition~\ref{prop:presprefixpoint}, it follows that
  $j_{e,f}(x)$  is a \pfp
  of $j_{i}$. Recall now that
  \ppsmall
  \begin{align*}
    x & = \phiPi(b) \land \aibi\,,\\
    j_{e,f}(y) & = j_{a_{k} \land b_{k} \land \aibi, b_{k} \land
      \aibi}(c) \land \aibi = J_{k}(y) \land \aibi\,, \quad\text{for
      each $y \in H$\,.}
  \end{align*}
  \ppnormal
  We have therefore
  \begin{align*}
    j_{i}(J_{k}(\phiPi(b))& = j_{i}(J_{k}(\phiPi(b) \land
    \aibi) \land \aibi) \land \aibi
    \\
    & =  j_{i}(j_{e,f}(\phiPi(b) \land \aibi)) \\
    & \leq j_{e,f}(\phiPi(b) \land \aibi) \leq
    J_{k}(\phiPi(b))\,,
  \end{align*}
  proving relation~\eqref{eq:ik}.
\end{proof}

\begin{remark}
  The above upper bound is tight. Recall that $I$ is the index set
  over the disjunction by which $\phiPi$ is defined, see
  \eqref{eq:defphipi}, so $\card(I)$ is the number of \atop{s} being
  joined.
  Computations with \fCube \cite{fCube} show that $\cl(\phiPi) = 2$
  when $\card(I) = 1$, and that $\cl(\phiPi) = 3$ when
  $\card(I) \in \set{2,3,4,5}$. Quite interestingly we obtained the
  same pattern for Ruitenburg's numbers: $\rho(\phiPi) = \cl(\phiPi)$
  when $\card(I) \in \set{2,3,4,5}$.  This raises the question whether
  the results presented in this section can be lifted to Ruitenburg's
  number; more generally and also considering
  Remark~\ref{rem:clstriclylessrho}, the question whether there is any
  formula $\phi \in \FIpc$ for which $\cl(\phi) < \rho(\phi)$ is open.
\end{remark}

\ifPreprint{
  \bibliographystyle{abbrv}
}{
  \bibliographystyle{ACM-Reference-Format}
}
\bibliography{biblio}

\end{document}